\documentclass[11pt,f letterpaper]{article}
\usepackage[german, french, american]{babel}
\usepackage[latin1]{inputenc}
\usepackage[T1]{fontenc}
\usepackage{amsmath, amssymb, amsfonts}
\usepackage{amsthm}
\usepackage{mathtools}
\usepackage{mathrsfs}
\usepackage{enumerate}
\usepackage{graphicx}
\usepackage{color}

\usepackage{epstopdf}

\usepackage{thumbpdf}
\usepackage{hyperref}
\usepackage[margin = 1in]{geometry}

\theoremstyle{plain}
\newtheorem{theorem}{Theorem}[section]
\newtheorem{lemma}[theorem]{Lemma}
\newtheorem{remark}[theorem]{Remark}
\newtheorem{corollary}[theorem]{Corollary}

\newtheorem{assumption}[theorem]{Assumption}

\newcommand{\e}{\mathbf{E}}
\newcommand{\h}{\mathbf{H}}
\newcommand{\z}{\mathbf{Z}}
\newcommand{\zz}[1]{\left. \mathbf{Z} \right|_{s={#1}}}
\newcommand{\dx}{\mathrm{d}\mathbf{x}}
\renewcommand{\d}{\mathrm{d}}
\newcommand{\g}{\mathbf{g}}
\newcommand{\n}{\boldsymbol{\nu}}
\newcommand{\m}{\mathbf{m}}
\newcommand{\x}{\mathbf{x}}
\newcommand{\opA}{\mathcal{B}}
\renewcommand{\a}{\mathfrak{a}}

\renewcommand{\Re}{\operatorname{Re}}

\newcommand{\DIV}{\mathop{\operatorname{div}}}
\newcommand{\CURL}{\mathop{\operatorname{curl}}}

\numberwithin{equation}{section}

\allowdisplaybreaks

\begin{document}

\title{Global Well-Posedness and Exponential Stability for \\
Heterogeneous Anisotropic Maxwell's Equations under a Nonlinear Boundary Feedback with Delay}

\author{Andrii Anikushyn\thanks{Department of Computer Sciences and Cybernetics, Taras Shevcheno National University of Kyiv \hfill \texttt{anik\_andrii@univ.kiev.ua}} \and
Michael Pokojovy\thanks{Department of Mathematical Sciences, The University of Texas at El Paso, El Paso, TX \hfill \texttt{mpokojovy@utep.edu}}}

\date{\today}

\maketitle

\begin{abstract}
    We consider an initial-boundary value problem for the Maxwell's system in a bounded domain
    with a linear inhomogeneous anisotropic instantaneous material law
    subject to a nonlinear Silver--M\"uller-type boundary feedback mechanism
    incorporating both an instantaneous damping and a time-localized delay effect.
    By proving the maximal monotonicity property of the underlying nonlinear generator,
    we establish the global well-posedness in an appropriate Hilbert space.
    Further, under suitable assumptions and geometric conditions, we show the system is exponentially stable.
\end{abstract}

\begin{center}
\begin{tabular}{p{1.0in}p{5.0in}}
    \textbf{Key words:} &  Maxwell's equations, nonlinear boundary feedback, instantaneous damping, time-localized delay, well-posedness, exponential stability  \\
    \textbf{MSC (2010):} & Primary
    35Q61, 
    35L50, 
    35L60, 
    35B40, 
    39B99  
    \\
    &
    Secondary
    35A01, 
    35A02, 
    35B37, 
    93C20, 
    93C23  
\end{tabular}
\end{center}

\section{Introduction}
Consider the macroscopic formulation of Maxwell's equations in a bounded domain $G \subset \mathbb{R}^{3}$
with $\boldsymbol{\nu} \colon \Gamma \to \mathbb{R}^{3}$ standing for the outer normal vector to its smooth boundary $\Gamma := \partial G$
and the functions $\mathbf{E}, \mathbf{D}, \mathbf{H}, \mathbf{B} \colon [0, \infty) \times G \to \mathbb{R}^{3}$
denoting the electric, displacement, magnetic and magnetizing fields, respectively.
With $\rho \colon [0, \infty) \times G \to \mathbb{R}$ representing the electric charge density,
Gauss' law along with Gauss' law for magnetism yield
\begin{equation}
    \DIV \mathbf{D} = \rho \quad \text{ and } \quad \DIV \mathbf{B} = 0 \quad \text{ in } (0, \infty) \times G,
    \label{EQUATION_GAUSS_LAWS}
\end{equation}
while Faraday's law of induction and Amp\`{e}re's circuital law mandate
\begin{equation}
    \partial_{t} \mathbf{D} = \CURL \mathbf{H} - \mathbf{J} \quad \text{ and } \quad
    \partial_{t} \mathbf{B} = -\CURL \mathbf{E} \quad \text{ in } (0, \infty) \times G.
    \label{EQUATION_FARADAY_AND_AMPERE_LAWS}
\end{equation}
Typically, $\mathbf{J} \colon [0, \infty) \times G \to \mathbb{R}^{3}$ is a (given) total current density.

Since the system (\ref{EQUATION_GAUSS_LAWS})--(\ref{EQUATION_FARADAY_AND_AMPERE_LAWS}) is underdetermined,
two more equations relating the four unknown vector fields $\mathbf{E}, \mathbf{D}, \mathbf{H}, \mathbf{B}$ need to be postulated.
Letting $\boldsymbol{\varepsilon}, \boldsymbol{\mu} \colon G \to \mathbb{R}^{3 \times 3}$ be symmetric, uniformly positive definite
matrix-valued permittivity and permeability tensor fields,
the instantaneous anisotropic material laws read as
\begin{equation}
    \mathbf{D} = \boldsymbol{\varepsilon} \mathbf{E} \quad \text{ and } \quad
    \mathbf{B} = \boldsymbol{\mu} \mathbf{H}.
    \label{EQUATION_MATERIAL_LAWS}
\end{equation}

Combining Equations (\ref{EQUATION_GAUSS_LAWS})--(\ref{EQUATION_MATERIAL_LAWS}), we arrive at
\begin{align}
    \partial_{t} \big(\boldsymbol{\varepsilon} \mathbf{E}\big) &= \phantom{-}\CURL \mathbf{H} - \mathbf{J}, &
    \DIV \big(\boldsymbol{\varepsilon} \mathbf{E}\big) &= \rho & \text{ in } (0, \infty) \times G,
    \label{EQUATION_MAXWELL_GENERAL_1} \\
    \partial_{t} \big(\boldsymbol{\mu} \mathbf{H}\big) &= -\CURL \mathbf{E}, &
    \DIV \big(\boldsymbol{\mu} \mathbf{H}\big) &= 0 & \text{ in } (0, \infty) \times G.
    \label{EQUATION_MAXWELL_GENERAL_2}
\end{align}

Various boundary conditions for Equations (\ref{EQUATION_MAXWELL_GENERAL_1})--(\ref{EQUATION_MAXWELL_GENERAL_2}) are known in the literature.
Eller {\it et al.} \cite{ElLaNi2002} considered the nonlinear version
\begin{equation}
    \mathbf{H} \times \boldsymbol{\nu} + \mathbf{g}(\mathbf{E} \times \boldsymbol{\nu}) \times \boldsymbol{\nu} = \mathbf{0}
    \quad \text{ on } (0, \infty) \times \Gamma
    \label{EQUATION_NONLINEAR_SILVER_MUELLER_BC}
\end{equation}
of the classical Silver--M\"uller boundary condition
\begin{equation}
    \mathbf{H} \times \boldsymbol{\nu} + \kappa \cdot (\mathbf{E} \times \boldsymbol{\nu}) \times \boldsymbol{\nu} = \mathbf{0}
    \quad \text{ on } (0, \infty) \times \Gamma.
    \label{EQUATION_LINEAR_SILVER_MUELLER_BC}
\end{equation}
Here, $\mathbf{g} \colon \mathbb{R}^{3} \to \mathbb{R}^{3}$ is a smooth function with $\mathbf{g}(\mathbf{0}) = \mathbf{0}$ and $\kappa > 0$ is a constant.
Equations (\ref{EQUATION_NONLINEAR_SILVER_MUELLER_BC}) and (\ref{EQUATION_LINEAR_SILVER_MUELLER_BC}) model scattering of electromagnetic waves by an obstacle $G$
under the assumption that the waves cannot penetrate the obstacle too deeply \cite[p. 20]{CaCoMo2011}.
The Silver--M\"uller boundary condition (\ref{EQUATION_LINEAR_SILVER_MUELLER_BC})
arises as a first-order approximation to the so-called transparent boundary condition
but, despite of being dissipative, allows for reflections back into the domain $G$ \cite[p. 136]{ElLaNi2002.2}.

In the present paper, we modify the nonlinear feedback-type boundary condition (\ref{EQUATION_NONLINEAR_SILVER_MUELLER_BC})
by incorporating a nonlinear time-localized delay effect:
\begin{equation}
    \mathbf{H}(t, \cdot) \times \boldsymbol{\nu}
    + \gamma_{1} \mathbf{g}\big(\mathbf{E}(t, \cdot) \times \boldsymbol{\nu}) \times \boldsymbol{\nu}
    + \gamma_{2} \mathbf{g}\big(\mathbf{E}(t - \tau, \cdot) \times \boldsymbol{\nu}) \times \boldsymbol{\nu} = \mathbf{0}
    \quad \text{ on } (0, \infty) \times \Gamma
    \label{EQUATION_NONLINEAR_SILVER_MUELLER_BC_WITH_DELAY}
\end{equation}
with a delay parameter $\tau > 0$ and appropriate constants $\gamma_{1}, \gamma_{2} > 0$.
Viewing the instantaneous Silver--M\"uller boundary conditions (\ref{EQUATION_NONLINEAR_SILVER_MUELLER_BC}) and (\ref{EQUATION_LINEAR_SILVER_MUELLER_BC})
as a feedback boundary control, the latter being a common stabilization instrument widely used in engineering,
an extra delay term in Equation (\ref{EQUATION_NONLINEAR_SILVER_MUELLER_BC_WITH_DELAY}) becomes indispensable
to adequately account for time retardations, which inevitably arise due a time lag in the interaction between a sensor measuring $\mathbf{E} \times \boldsymbol{\nu}$
and the actuator updating $\mathbf{H} \times \boldsymbol{\nu}$ on the boundary $\Gamma$.

Pulling Equations (\ref{EQUATION_MAXWELL_GENERAL_1})--(\ref{EQUATION_MAXWELL_GENERAL_2}),
(\ref{EQUATION_NONLINEAR_SILVER_MUELLER_BC_WITH_DELAY}) together, we arrive at
\begin{align*}
    \partial_{t} \big(\boldsymbol{\varepsilon} \mathbf{E}\big) = \CURL \mathbf{H} - \mathbf{J}, \quad
    \DIV \big(\boldsymbol{\varepsilon} \mathbf{E}\big) &= \rho \quad \text{ in } (0, \infty) \times G, \\
    \partial_{t} \big(\boldsymbol{\mu} \mathbf{H}\big) = -\CURL \mathbf{E}, \quad
    \DIV \big(\boldsymbol{\mu} \mathbf{H}\big) &= 0 \quad \text{ in } (0, \infty) \times G, \\
    \mathbf{H}(t, \cdot) \times \boldsymbol{\nu}
    + \gamma_{1} \mathbf{g}\big(\mathbf{E}(t, \cdot) \times \boldsymbol{\nu}) \times \boldsymbol{\nu}
    + \gamma_{2} \mathbf{g}\big(\mathbf{E}(t - \tau, \cdot) \times \boldsymbol{\nu}) \times \boldsymbol{\nu} &= \mathbf{0}
    \quad \text{ on } (0, \infty) \times \Gamma.
\end{align*}
In the following, let $\mathbf{J} \equiv 0$ and $\rho \equiv 0$.
This corresponds to the case both electrical sources and resistance effects are absent.
While not affecting the well-posedness results to follow, compared to the case of electrical resistance, i.e., $\mathbf{J} = \boldsymbol{\sigma}(\mathbf{E}, \mathbf{H}) \mathbf{E}$ as mandated by the Ohm's law,
the condition $\mathbf{J} \equiv 0$ reduces the overall amount of damping in the system making the stability analysis more challenging.
Adding the usual initial conditions, we arrive at the system
\begin{align}
    \partial_{t} \big(\boldsymbol{\varepsilon} \mathbf{E}\big) = \phantom{-}\CURL \mathbf{H}, \quad
    \DIV \big(\boldsymbol{\varepsilon} \mathbf{E}\big) &= 0 \quad \text{ in } (0, \infty) \times G,
    \label{EQUATION_MAXWELL_1} \\
    \partial_{t} \big(\boldsymbol{\mu} \mathbf{H}\big) = -\CURL \mathbf{E}, \quad
    \DIV \big(\boldsymbol{\mu} \mathbf{H}\big) &= 0 \quad \text{ in } (0, \infty) \times G,
    \label{EQUATION_MAXWELL_2} \\
    \mathbf{H}(t, \cdot) \times \boldsymbol{\nu}
    + \gamma_{1} \mathbf{g}\big(\mathbf{E}(t, \cdot) \times \boldsymbol{\nu}) \times \boldsymbol{\nu}
    + \gamma_{2} \mathbf{g}\big(\mathbf{E}(t - \tau, \cdot) \times \boldsymbol{\nu}) \times \boldsymbol{\nu} &= \mathbf{0}
    \quad \text{ on } (0, \infty) \times \Gamma,
    \label{EQUATION_MAXWELL_BC_1} \\
    \mathbf{E}(0, \cdot) = \mathbf{E}^{0}, \quad
    \mathbf{H}(0, \cdot) &= \mathbf{H}^{0} \quad \text{ in } G,
    \label{EQUATION_MAXWELL_IC} \\
    \mathbf{E}(-\tau \cdot, \cdot) \times \boldsymbol{\nu} &= \boldsymbol{\Phi}^{0} \quad \text{ in } (0, 1) \times \Gamma.
    \label{EQUATION_MAXWELL_IC_HISTORY}
\end{align}

Partial (not mentioning ordinary!) differential equations (PDEs) have widely been studied in the literature.
Time-delays along with other types of time-nonlocalities such as memory effects, etc.,
can typically enter a PDE in one of the two ways -- either through a time-nonlocal material law \cite{Jo2003, KhuPoRa2015}
or a time-delayed feedback mechanism (so-called ``closed-loop control'') \cite{Da1997, DaLaPo1986, NiPo2005, NiPi2006.SIAM, ZhaSte2017}, etc.
Whereas time-delayed material laws mostly lead to ill-posedness \cite{KhuPoRa2015},
the effect of time-delay in feedback mechanisms can range from a ``mere'' reduction of the decay rate to destabilization to even ill-posedness.
We refer the reader to the famous Datko's example \cite{Da1997}, which illustrates the later dichotomy.
Our goal is to investigate the impact of the nonlinear boundary delay feedback from Equation (\ref{EQUATION_MAXWELL_BC_1})
on system (\ref{EQUATION_MAXWELL_1})--(\ref{EQUATION_MAXWELL_IC_HISTORY}).
Before proceeding with our study, we first give a short literature review.
In our brief review below, we restrict ourselves to instantaneous material laws but discuss both instantaneous and nonlocal boundary conditions.


Lagnese \cite{La1989} studied the exact boundary controllability of homogeneous isotropic Maxwell's equations
\begin{align}
    \label{INTRO_LA_SYSTEM}
    \partial_{t} \e - \CURL( \varepsilon^{-1} \h) = \mathbf{0}, \quad  \partial_{t} \h + \CURL(\mu^{-1} \e) &= \mathbf{0} \text{ in } (0, \infty) \times G, \\
    \DIV \e = \DIV \h &= \mathbf{0} \text{ in } (0, \infty) \times G, \\
    \label{INTRO_LA_IC}
    \e(0) = \e^0, \quad \h(0) &= \h^0 \text{ in } G
\end{align}
subject to boundary condition
\begin{equation}
    \label{INTRO_LA_BC}
    \n \times \h = - \mathbf{J} \text{ on } \Gamma \times (0, \infty)
\end{equation}
in star-shaped regions $G$.
Here, the current density $\mathbf{J}$ plays the role of a distributed open-loop control.
The electric permittivity $\varepsilon$ and magnetic permeability $\mu$ were assumed constant,
while the region $G$ was selected to be star-shaped with respect to some point.

Nicaise \cite{Ni2000} investigated the exact controllability of isotropic non-homogeneous Maxwell's equations (\ref{EQUATION_MAXWELL_1})--(\ref{EQUATION_MAXWELL_2}) with the boundary conditions
\begin{align*}
    \h \times \n &= \mathbf{J} \text{ on } (0, T) \times \Gamma_0, \\
    \h \times \n &= \mathbf{0} \text{ on } (0, T) \times (\Gamma \setminus \Gamma_0),  \\
    \e \times \n &= \mathbf{0} \text{ on } (0, T) \times \Gamma
\end{align*}
via a boundary control $\mathbf{J}$ under appropriate conditions on the coefficients and the geometry of $G$. 
Here, $\Gamma_0$ is a non-empty, relatively open subset of $\Gamma$.

Eller \& Masters \cite{ElMa2002} later used multiplier techniques to prove the exact controllability for Equations (\ref{INTRO_LA_SYSTEM})--(\ref{INTRO_LA_IC}) via of the boundary control
\begin{align*}
    \n \times (\varepsilon^{-1} \h) = \mathbf{J} \text{ on } (0, T) \times \Gamma
\end{align*}
for nonhomogeneous $\mu, \varepsilon$ in connected domains $G$.

Krigman \cite{Kri2007} studied a similar problem for the system
\begin{align*}
    \varepsilon \partial_{t} \e - \CURL( \h) + \sigma \e = \mathbf{0}, \quad \mu \partial_{t} \h + \CURL(\e) &= \mathbf{0} \text{ in } (0,\infty) \times G \\
    \DIV (\boldsymbol{\varepsilon} \e) = 0, \quad \DIV (\boldsymbol{\mu} \h) &= 0 \text{ in } (0,\infty) \times G
\end{align*}
with the initial conditions (\ref{INTRO_LA_IC}) and boundary condition (\ref{INTRO_LA_BC}) in simply connected star-shaped domains $G$.

Eller \cite{El2007} studied Equations (\ref{EQUATION_MAXWELL_1})--(\ref{EQUATION_MAXWELL_2}) subject to boundary conditions
\begin{align*}
    \n\times \e = \mathbf{0}, \quad \n\cdot(\boldsymbol{\mu} \h) = 0 \text{ on } (0, T) \times \Gamma.
\end{align*}
Assuming the star-shapedness of $G$ and exploiting the method of multipliers, a boundary observability inequality was proved.


Eller {\it et al.} \cite{ElLaNi2002} examined the problem of stabilizing Maxwell's equations (\ref{EQUATION_MAXWELL_1})--(\ref{EQUATION_MAXWELL_2}) subject to boundary condition
\begin{align*}
    \h \times \n + g(\e \times \n) \times \n = \mathbf{0} \text{ on } (0, \infty) \times \Gamma.
\end{align*}
The (scalar) $\varepsilon$ and $\mu$ were assumed real positive fields and $g(\cdot)$ a continuous mapping satisfying certain monotonicity and boundness conditions.
To prove the well-posedness, monotone operator theory and nonlinear semigroup theory were used,
while the exponential stability -- both in the linear and the nonlinear cases -- was shown via exact controllability established using multiplier techniques.

Zhou \cite{Zh1997} investigated the exact controllability under the action of a distributed control $\mathbf{u}$
\begin{align*}
    \partial_{t} \mathbf{E} - \CURL (\mathbf{H}) &= \chi_G(\x) \mathbf{u} \text{ in } (0, \infty) \times G, \\
    \partial_{t} \mathbf{H} + \CURL ( \mathbf{E}) &= \mathbf{0} \text{ in } (0,\infty) \times G, \\
    \DIV \mathbf{H} = \DIV \mathbf{E} &= 0 \text{ in } (0, \infty) \times G, \\
    \mathbf{E} \times \n &= \mathbf{0} \text{ on } (0, \infty) \times G, \\
    \mathbf{E}(0) = \mathbf{E}^0, \quad \mathbf{H}(0) &= \mathbf{H}^0 \text{ in } G,
\end{align*}
where $\chi_G(\cdot)$ is the indicator function of a set $\omega \subset G$.
This result was further extended by Zhang \cite{Zha2000} to time-dependent $\omega$'s using multiplier techniques.

A series of important results were obtained by Nicaise \& Pignotti.
In \cite{NiPi2003}, under monotonicity and boundedness assumptions on $g(\cdot)$, the authors considered a stabilization problem for Maxwell's equations
\begin{align}
    \label{INTRO_NP_SYSTEM_1}
    \partial_{t} \mathbf{E} - \CURL (\lambda \mathbf{H}) &= \mathbf{0} \text{ in } (0, \infty) \times G, \\
    \label{INTRO_NP_SYSTEM_2}
    \partial_{t} \mathbf{H} + \CURL (\boldsymbol{\mu} \mathbf{E}) &= \mathbf{0} \text{ in } (0, \infty) \times G, \\
    \label{INTRO_NP_SYSTEM_3}
    \DIV \mathbf{E} = \DIV \mathbf{H} &= 0 \text{ in } (0, \infty) \times G, \\
    \label{INTRO_NP_SYSTEM_4}
    \mathbf{E}(0) = \mathbf{E}^0, \quad \mathbf{H}(0) &=  \mathbf{H}^0 \text{ in } G
\end{align}
with space-time variable (scalar) coefficients $\boldsymbol{\mu} = \boldsymbol{\mu}(\x, t)$, $\lambda = \lambda(\x, t)$  and a nonlinear Silver--M\"{u}ller boundary condition
\begin{align*}
    g(\mathbf{x}, \mathbf{E} \times \n) + \mathbf{H} \times \n = \mathbf{0} \text{ on } (0, \infty) \times \Gamma.
\end{align*}
Another article \cite{NiPi2004} by the same authors was dedicated to the problem of stabilization of Maxwell's equations via a distributed feedback arising from the linear Ohm's law:
\begin{align}
    \label{INTRO_NP_SYSTEM_7}
    \partial_{t} \mathbf{E} - \CURL (\lambda \mathbf{H}) + \sigma \mathbf{E} &= \mathbf{0} \text{ in } (0, \infty) \times G, \\
    \label{INTRO_NP_SYSTEM_8}
    \partial_{t} \mathbf{H} + \CURL (\boldsymbol{\mu} \mathbf{E}) &= \mathbf{0} \text{ in } (0, \infty) \times G, \\
    \label{INTRO_NP_SYSTEM_9}
    \DIV \mathbf{H} &= 0 \text{ in } (0, \infty) \times G, \\
    \label{INTRO_NP_SYSTEM_5}
    \mathbf{E} \times \n = \mathbf{0}, \quad \mathbf{H} \cdot \n &= 0 \text{ on } (0, \infty) \times \Gamma, \\
    \label{INTRO_NP_SYSTEM_6}
    \mathbf{E}(0) = \e^0, \quad \mathbf{H}(0) &= \h^0 \text{ in } G.
\end{align}
The method of multipliers was used to establish an observability estimate in the paper.
Same authors \cite{NiPi2006.AMO} also obtained an observability estimate for the standard isotropic homogeneous Maxwell's system (\ref{INTRO_NP_SYSTEM_1})--(\ref{INTRO_NP_SYSTEM_4})
subject to boundary conditions (\ref{INTRO_NP_SYSTEM_5}).

The impact of boundary conditions that include tangetial components were studied by numerous authors.
Kapitonov \cite{Ka1994} considered Equations (\ref{INTRO_NP_SYSTEM_1})--(\ref{INTRO_NP_SYSTEM_4}) in $(0, T) \times G$ with dissipative boundary conditions
\begin{align*}
    \n \times \e - \alpha(\cdot) \h_{\tau} = \mathbf{0},
\end{align*}
where $\alpha(\cdot)$ is a continuously differentiable function on $\Gamma$ with $\Re \alpha > 0$.
Here and in the sequel,
\begin{equation}
    \label{EQUATION_TANGENTIAL_COMPONENT}
    \h_{\tau} := \h - (\h \times \n) \h
\end{equation}
denotes the tangential component of $\h$.
Using the semigroup approach to investigate the well-posedness, the author further utilized geometrical properties of the domain
to obtain results on exact boundary controllability of the solution to (\ref{INTRO_NP_SYSTEM_1})--(\ref{INTRO_NP_SYSTEM_4}) in $(0, T) \times G$ with boundary condition
\begin{align*}
    \n \times \e - \mathrm{i} a(\x) \h_{\tau} |_{\Gamma} = \mathbf{p}(t, \x),
\end{align*}
where $a(\x)$ is a continuously differentiable scalar function on $\Gamma$.
Cagnol and Eller \cite{CaEl2011} studied a well-posedness for anisotropic Maxwell's equations with the so-called ``absorbing boundary'' condition
\begin{align*}
    \n \times \e - \alpha \h_{\tau} = \mathbf{g} \text{ on } (0, T) \times \Gamma.
\end{align*}

Nonlocal boundary conditions are also known in the literature.
Nibbi \& Polidoro \cite{NiPo2005} proved the exponential stability of `Graffi'-type free energy associated with the isotropic Maxwell's equations subject to a memory-type boundary condition
\begin{align*}
    \e_{\tau}(t, \cdot) = \eta_0 \h(t, \cdot) \times \n + \int_0^\infty \eta(s) \h(t - s, \cdot) \times \n\ \d s.
\end{align*}
In contrast, the impact of time-delayed boundary conditions from Equation (\ref{EQUATION_MAXWELL_BC_1}) on Maxwell's equations has not been studied in the literature before.
At the same time, such boundary conditions proved to be very interesting -- both from theoretical in practical point of view -- for other types of hyperbolic systems.
For example, Nicaise \& Pignotti \cite{NiPi2006.SIAM} investigated the stability of a delay wave equation subject to a time-delayed boundary feedback
\begin{align*}
    u_{tt}(t, \x)  - \Delta u(t, \x) = 0 &\text{ for } (t, \x) \in (0, \infty) \times G, \\
    u(t, \x) = 0 &\text{ for } (t, \x) \in (0, \infty) \times \Gamma_{0}, \\
    \partial_{\n} u(t, \x) = -\mu_1 u_t(t, \x) - \mu_2 u(t - \tau, \x) &\text{ for } (t, \x) \in (0, \infty) \times (\Gamma \setminus \Gamma_0), \\
    u(0, \x) = u_0(\x),\ \partial_{t} u(0, \x) = u_1(\x) &\text{ for } \x \in G, \\
    u_t(t-\tau, \x) = f_0(t-\tau, \x) &\text{ for } (t, \x) \in (0, \tau) \times (\Gamma \setminus \Gamma_0).
\end{align*}
Under suitable conditions on $\Gamma_0$, the initial-boundary-value problem was shown to possess a unique strong solution,
which is exponentially stable given $\mu_2 < \mu_1$.

Nicaise \& Pignotti \cite{NiPi2007} studied Equations (\ref{EQUATION_MAXWELL_1})--(\ref{EQUATION_MAXWELL_IC_HISTORY})
subject to the linear feedback $g(\mathbf{x}) = \mathbf{x}$ in the boundary condition (\ref{EQUATION_MAXWELL_BC_1})
and a similar system subject to an internal delay feedack
\begin{align}
    \partial_{t} \big(\boldsymbol{\varepsilon} \mathbf{E}\big) - \CURL \mathbf{H} + \sigma\big(\mu_{1} \mathbf{E} + \mu_{2} \mathbf{E}(\cdot - \tau, \cdot)\big) &= 0, &
    \DIV \big(\boldsymbol{\varepsilon} \mathbf{E}\big) &= 0 \quad \text{ in } (0, \infty) \times G
\end{align}
in lieu of Equation (\ref{EQUATION_MAXWELL_1}) with homogeneous, isotropic $\boldsymbol{\varepsilon}$, $\boldsymbol{\mu}$.
Having established a well-posedness theory for both systems for the case $\mu_{2} \leq \mu_{1}$,
exponential stability in simply connected bounded domains for $\mu_{2} < \mu_{1}$ was proved
and optimality of the latter condition was illustrated for both systems.

The rest of the paper has the following outline.
In Section \ref{SECTION_WELL_POSEDNESS}, partial difference-differential Equations (\ref{EQUATION_MAXWELL_1})--(\ref{EQUATION_MAXWELL_IC_HISTORY}) are transformed
to an abstact nonlinear evolution equation on the extended phase space.
By showing maximal monotonicity of the generator and exploiting the nonlinear semigroup theory, the well-posedness is proved.
In Section \ref{SECTION_EXPONENTIAL_STABILITY}, under a star-shapedness assumption on the domain $G$,
the exponential stability of the system is shown by using standard Rellich's multipliers and auxiliary functions inspired by \cite{KhuPoRa2015}.
In the Appendix section, for the sake of completeness, a ``folklore'' method (which probably goes back to early works of I.~Lasiecka)
that establishes a connection between disipativity, an observability-through-damping inequality and exponential stability is formulated and proved.


\section{Well-Posedness}
\label{SECTION_WELL_POSEDNESS}
Following \cite{El2007}, for a symmetric, positive definite matrix-valued $\boldsymbol{\alpha} \in L^{\infty}(G, \mathbb{R}^{3 \times 3})$, we define the spaces
\begin{align*}
    H(\CURL, G) &:= \Big\{\mathbf{u} \in \big(L^{2}(G)\big)^{3} \,\big|\, \CURL \mathbf{u} \in \big(L^{2}(G)\big)^{3}\Big\}, \\
    H(\operatorname{div}_{\boldsymbol{\alpha}} 0, G) &:=
    \Big\{\mathbf{u} \in \big(L^{2}(G)\big)^{3} \,\big|\, \DIV (\boldsymbol{\alpha}\mathbf{u}) = 0\Big\} \notag
\end{align*}
and introduce the Hilbert space
\begin{equation}
    \mathcal{H} := H(\operatorname{div}_{\boldsymbol{\varepsilon}} 0, G) \times H(\operatorname{div}_{\boldsymbol{\mu}} 0, G) \notag
\end{equation}
endowed with the inner product
\begin{equation}
    \big\langle (\mathbf{E}, \mathbf{H})^{T}, (\tilde{\mathbf{E}}, \tilde{\mathbf{H}})^{T}\big\rangle_{\mathcal{H}} :=
    \int_{G} \boldsymbol{\varepsilon} \mathbf{E} \cdot \tilde{\mathbf{E}} \,\mathrm{d}\mathbf{x} +
    \int_{G} \boldsymbol{\mu} \mathbf{H} \cdot \tilde{\mathbf{H}} \,\mathrm{d}\mathbf{x}. \notag
\end{equation}
(The completeness follows from \cite{La1989}).

Similar to \cite{ElLaNi2002}, we formally define the operator
\begin{equation}
    \mathcal{A} \colon
    \begin{pmatrix}
        \mathbf{E} \\
        \mathbf{H}
    \end{pmatrix} \mapsto
    \begin{pmatrix}
        -\boldsymbol{\varepsilon}^{-1} \CURL \mathbf{H} \\
        \phantom{-}\boldsymbol{\mu}^{-1} \CURL \mathbf{E}
    \end{pmatrix}.
    \notag
\end{equation}
Our goal is to transform Equations (\ref{EQUATION_MAXWELL_1})--(\ref{EQUATION_MAXWELL_IC_HISTORY})
to an abstract Cauchy problem on the extended phase space (cf. \cite{KhuPoRa2015, NiPi2006.SIAM, NiPi2007})
\begin{equation}
    \mathscr{H} := \mathcal{H} \times L^{2}\big(0, 1; L^{2}_{\tau}\big(\Gamma, \mathbb{R}^{3}\big)\big) \notag
\end{equation}
endowed with the scalar product
\begin{align*}
    \big\langle (\mathbf{E}, \mathbf{H}, \mathbf{Z})^{T}, (\tilde{\mathbf{E}}, \tilde{\mathbf{H}}, \tilde{\mathbf{Z}})^{T}\big\rangle_{\mathscr{H}}
    &:=
    \big\langle (\mathbf{E}, \mathbf{H})^{T}, (\tilde{\mathbf{E}}, \tilde{\mathbf{H}})^{T}\big\rangle_{\mathcal{H}}
    + \tau \int_{0}^{1} \int_{\Gamma} \big(\mathbf{Z}(s, \cdot) \times \boldsymbol{\nu}\big) \cdot
    \big(\tilde{\mathbf{Z}}(s, \cdot) \times \boldsymbol{\nu}\big) \mathrm{d}\mathbf{x} \mathbf{d}s.
\end{align*}
Here, following \cite[Equation (2.12)]{NiPi2007}, the `tangential' $L^{2}$-space on $\Gamma$ $L^{2}_{\tau}$ is defined as
\begin{equation*}
    L^{2}_{\tau}\big(\Gamma, \mathbb{R}^{3}\big) :=
    \big\{\boldsymbol{\Psi} \in L^{2}(\Gamma, \mathbb{R}^{3}) \,|\, \boldsymbol{\Psi} \cdot \boldsymbol{\nu} = 0 \text{ on } \Gamma\big\} .
\end{equation*}
This choice will later prove crucial for showing the density of the generator.

Letting formally
\begin{equation}
    \mathbf{V}(t, \cdot) :=
    \begin{pmatrix}
        \mathbf{E}(t, \cdot) \\
        \mathbf{H}(t, \cdot) \\
        (0, 1) \ni s \mapsto \mathbf{Z}(t, s, \cdot)
    \end{pmatrix}
    \equiv
    \begin{pmatrix}
        \mathbf{E}(t, \cdot) \\
        \mathbf{H}(t, \cdot) \\
        (0, 1) \ni s \mapsto \big(\mathbf{E}(t - \tau s) \times \boldsymbol{\nu}\big)|_{\Gamma})
    \end{pmatrix},
    \notag
\end{equation}
we define the operator
\begin{equation*}
    \mathscr{A} \colon D(\mathscr{A}) \subset \mathscr{H} \to \mathscr{H}, \quad
    (\mathbf{E}, \mathbf{H}, \mathbf{Z})^{T}
    \mapsto
    \begin{pmatrix}
        \mathcal{A} (\mathbf{E}, \mathbf{H})^{T} \\
        \frac{1}{\tau} \partial_{s} \mathbf{Z}
    \end{pmatrix}
\end{equation*}
with the domain
\begin{equation}
    \begin{split}
        D(\mathscr{A}) :=
        \Big\{(\mathbf{E}, \mathbf{H}, \mathbf{Z})^{T} \in \mathscr{H} \,\big|\,
        &\big(\mathcal{A} (\mathbf{E}, \mathbf{H})^{T}, \tfrac{1}{\tau} \partial_{s} \mathbf{Z}\big)^{T} \in \mathscr{H}, \quad
        \mathbf{E} \times \boldsymbol{\nu}|_{\Gamma}, \mathbf{H} \times \boldsymbol{\nu}|_{\Gamma} \in \big(L^{2}(\Gamma)\big)^{3}, \\
        &\mathbf{H} \times \boldsymbol{\nu} + \gamma_1 \mathbf{g}(\mathbf{E} \times \boldsymbol{\nu}) \times \boldsymbol{\nu} +
        \gamma_{2} \mathbf{g}\big(\mathbf{Z}|_{s = 1} \big) \times \boldsymbol{\nu} = \mathbf{0} \text{ on } \Gamma, \\
        &\mathbf{Z}|_{s = 0} = \mathbf{E} \times \boldsymbol{\nu}
        \Big\}.
    \end{split}
    \notag
\end{equation}
The latter explicitly reads as
\begin{equation}
    \begin{split}
        D(\mathscr{A}) :=
        \Big\{(\mathbf{E}, \mathbf{H}, \mathbf{Z})^{T} \in \mathscr{H} \,\big|\,
        &\mathbf{E}, \mathbf{H} \in H(\CURL, G), \quad \mathbf{Z} \in H^{1}\big(0, 1; \big(L^{2}(\Gamma)\big)^{3}\big), \\
        &\mathbf{E} \times \boldsymbol{\nu}|_{\Gamma}, \mathbf{H} \times \boldsymbol{\nu}|_{\Gamma} \in \big(L^{2}(\Gamma)\big)^{3}, \\
        &\mathbf{H} \times \boldsymbol{\nu} + \gamma_1 \mathbf{g}(\mathbf{E} \times \boldsymbol{\nu}) \times \boldsymbol{\nu} +
        \gamma_{2} \mathbf{g}\big(\mathbf{Z}|_{s = 1} \big) \times \boldsymbol{\nu} = \mathbf{0} \text{ on } \Gamma, \\
        &\mathbf{Z}|_{s = 0} = \mathbf{E} \times \boldsymbol{\nu}
        \Big\}.
    \end{split}
    \notag
\end{equation}

Equations (\ref{EQUATION_MAXWELL_1})--(\ref{EQUATION_MAXWELL_IC_HISTORY}) can equivalently be written as an abstract evolution equation
\begin{equation}
    \partial_{t} \mathbf{V}(t) + \mathscr{A}\big(\mathbf{V}(t)\big) = \mathbf{0} \text{ for } t > 0, \quad
    \mathbf{V}(0) = \mathbf{V}^{0}
    \label{EQUATION_ABSTRACT_CAUCHY_PROBLEM}
\end{equation}
with $\mathbf{V}^{0} := (\mathbf{E}^{0}, \mathbf{H}^{0}, \boldsymbol{\Phi}^{0})^{T}$.

\begin{assumption}[Tensor fields $\boldsymbol{\varepsilon}$ and $\boldsymbol{\mu}$]
    \label{ASSUMPTION_BASIC_CONDITIONS_COEFFICIENTS}

    Let $\boldsymbol{\varepsilon}, \boldsymbol{\mu} \in C^{0}\big(\bar{G}, \mathbb{R}^{3 \times 3}\big)$ satisfy
    \begin{equation}
        \big(\boldsymbol{\varepsilon}(\mathbf{x})\big)^{T} = \boldsymbol{\varepsilon}(\mathbf{x}) \text{ and }
        \big(\boldsymbol{\mu}(\mathbf{x})\big)^{T} = \boldsymbol{\mu}(\mathbf{x}) \text{ for } \mathbf{x} \in \bar{G}
    \end{equation}
    as well as
    \begin{equation}
        \lambda_{\min}(\boldsymbol{\varepsilon}) > 0 \text{ and }
        \lambda_{\min}(\boldsymbol{\mu}) > 0, \notag
    \end{equation}
    where
    \begin{equation}
        \lambda_{\min}(\boldsymbol{\varphi }) :=
        \min_{x \in \bar{G}} \min_{|\boldsymbol{\xi}| = 1} \boldsymbol{\xi} \cdot \big(\boldsymbol{\varphi}(\mathbf{x}) \boldsymbol{\xi}\big)
        \text{ for } \boldsymbol{\varphi } \in C^{0}\big(\bar{G}, \mathbb{R}^{3 \times 3}\big).
        \notag
    \end{equation}
\end{assumption}

Denote
\begin{equation}
    \label{ASSUMPTIONS_ALPHA}
    \alpha = \min \{ \lambda_{\min}(\boldsymbol{\varepsilon}) , \lambda_{\min}(\boldsymbol{\mu})\}.
\end{equation}

\begin{assumption}[Nonlinearity $\mathbf{g}(\cdot)$]
    \label{ASSUMPTION_NONLINEARITY_G}

    Suppose the nonlinear function $\mathbf{g} \colon \mathbb{R}^{3} \to \mathbb{R}^{3}$ satisfies:
    \begin{enumerate}
        \item $\mathbf{g}(\mathbf{0}) = \mathbf{0}$,

        \item There exists $c_1 > 0$ such that $\big(\mathbf{g}(\mathbf{E}) - \mathbf{g}(\tilde{\mathbf{E}})\big) \cdot (\mathbf{E} - \tilde{\mathbf{E}}) \geq c_1 \big| \mathbf{E} - \tilde{\mathbf{E}} \big|^2$
        for any $\mathbf{E}, \tilde{\mathbf{E}} \in \mathbb{R}^{3}$,

        \item There exists $c_2 > 0$ such that $\big|\mathbf{g}(\mathbf{E}) - \mathbf{g}(\tilde{\mathbf{E}})\big| \leq c_2 |\mathbf{E}-\tilde{\mathbf{E}}|$ for any $\mathbf{E}, \tilde{\mathbf{E}} \in \mathbb{R}^{3}$.
    \end{enumerate}
\end{assumption}

\begin{remark}
    In contrast to the wave equation, which is known \cite{LaTa1993} to admit feedback functions with a superlinear growth rate (in $y$, not $y_{t}$),
    this is no longer true for Maxwell's equations since superlinear terms can cause the solution to leave the basic $L^{2}$-space thus destroying the well-posedness.
    In this sense, the results of our paper appear to be optimal -- at least at the basic energy level.
\end{remark}

The following two lemmas are quoted from \cite{ElLaNi2002}.

\begin{lemma}
    \label{LEMMA_GREENS_FORMULA}
    For all $\mathbf{E}, \mathbf{H} \in H(\CURL, G)$ with $\mathbf{E} \times \boldsymbol{\nu}|_{\Gamma}, \mathbf{H} \times \boldsymbol{\nu}|_{\Gamma} \in \big(L^{2}(\Gamma)\big)^{3}$, we have
    \begin{equation}
        \notag
        \int_G \big( \CURL \e \cdot \h  - \CURL \h \cdot \e \big) \dx = \int_{\Gamma} (\h \times \boldsymbol{\nu}) \cdot \e \, \d \x.
    \end{equation}
\end{lemma}

\begin{remark}
    By virtue of \cite[Chapter 2]{Ce96}, the statement of Lemma remains true for general fields $\mathbf{E}, \mathbf{H} \in H(\CURL, G)$
    if the latter integral is replaced by the dual paring between the spaces
    $H^{-1/2}(\operatorname{div}, \Gamma)$ and $H^{1/2}(\operatorname{curl}, \Gamma)$.
\end{remark}

Let $P_{\boldsymbol{\varepsilon}}$ denote the orthogonal projection on $H(\DIV_{\boldsymbol{\varepsilon}} 0, G)$ in $(L^2(G))^3$.
Combining \cite[Lemma 2.3]{ElLaNi2002} with the `tangentiality' of $\mathbf{Z} \in L^{2}_{\tau}\big(\Gamma, \mathbb{R}^{3}\big)$, we get:
\begin{lemma}
    \label{LEMMA_DENSITY}
    The image $P_{\boldsymbol{\varepsilon}} \big((\mathcal{D}(G))^3\big)$ is dense in $H(\DIV_{\boldsymbol{\varepsilon}} 0, G)$.
    The domain of the operator $\mathscr{A}$ is dense in $\mathscr{H}$.
\end{lemma}

\begin{remark}
    \label{REMARK_PROPERTIES_OF_P}
    For all $\boldsymbol{\chi} \in C^{\infty}(\bar{G})$, we have $\CURL (P_{\boldsymbol{\varepsilon}} \boldsymbol{\chi}) = \CURL \chi$ in $G$ and
    $(P_{\boldsymbol{\varepsilon}} \chi) \times \n = \boldsymbol{\chi} \times \n$ on $\Gamma$.
\end{remark}

Now, we can prove the following lemma.

\begin{lemma}
    \label{LEMMA_MONOTONICITY}
    There exists a positive number $C$ such that $C \operatorname{id} + \mathscr{A}$ is a maximal monotone operator.
\end{lemma}

\begin{proof}
\emph{Monotonicity: }
Consider a new inner product on $\mathscr{H}$ defined via
\begin{align*}
    \big\langle (\mathbf{E}, \mathbf{H}, \mathbf{Z})^{T}, (\tilde{\mathbf{E}}, \tilde{\mathbf{H}}, \tilde{\mathbf{Z}})^{T}\big\rangle_{\tilde{\mathscr{H}}}
    &:=
    \big\langle (\mathbf{E}, \mathbf{H})^{T}, (\tilde{\mathbf{E}}, \tilde{\mathbf{H}})^{T}\big\rangle_{\mathcal{H}} \\
    &+ \xi \tau \int_{0}^{1} \int_{\Gamma} e^{c s} \big(\mathbf{E}(t - \tau s) \times \boldsymbol{\nu}\big) \cdot
    \big(\tilde{\mathbf{E}}(t - \tau s) \times \boldsymbol{\nu}\big) \mathrm{d}\mathbf{x} \mathbf{d}s.
\end{align*}
Here $c, \xi$ are positive numbers and will be chosen later.
Obviously, $\big\langle \cdot, \cdot \big\rangle_{\tilde{\mathscr{H}}} $ is equivalent with the original inner product $\big\langle \cdot, \cdot \big\rangle_{\mathscr{H}}$.

First, we show that $C \operatorname{id} + \mathscr{A}$ is a monotone operator for some $C > 0$.
For all $(\e,\h,\z)^{T},$ $(\e',\h',\z')^{T} \in D(\mathscr{A})$, letting $(\tilde{\e},\tilde{\h},\tilde{\z})^{T} = (\e,\h,\z)^{T} - (\e',\h',\z')^{T}$, we obtain
\begin{align}
    \notag
& \Bigg\langle (C\operatorname{id} + \mathscr{A}) \begin{pmatrix} \e  \\ \h \\ \z   \end{pmatrix}  - (C \operatorname{id} + \mathscr{A}) \begin{pmatrix} \e'  \\ \h' \\ \z'  \end{pmatrix}, \begin{pmatrix} \e  \\ \h \\ \z  \end{pmatrix} - \begin{pmatrix} \e'  \\ \h' \\ \z'  \end{pmatrix}\Bigg\rangle_{\tilde{\mathscr{H}}} \\
    \notag
& = C \left\| \begin{pmatrix} \tilde{\e}  \\ \tilde{\h} \\ \tilde{\z}   \end{pmatrix}  \right\|^2_{\tilde{\mathscr{H}}} + \Bigg\langle  \mathscr{A} \begin{pmatrix} \tilde{\e}  \\ \tilde{\h} \\ \tilde{\z}  \end{pmatrix}  , \begin{pmatrix} \tilde{\e}  \\ \tilde{\h} \\ \tilde{\z}    \end{pmatrix} \Bigg\rangle_{\tilde{\mathscr{H}}}\\
    \notag
& = C \left\|   \begin{pmatrix} \tilde{\e}  \\ \tilde{\h} \\ \tilde{\z} \end{pmatrix} \right\|^2_{\tilde{\mathscr{H}}}
    + \Bigg\langle  \begin{pmatrix} -\boldsymbol{\varepsilon}^{-1} \CURL \tilde{\h}  \\ \boldsymbol{\mu}^{-1} \CURL \tilde{\e} \\ \tau^{-1} \partial_{s} \tilde{\z}    \end{pmatrix} , \begin{pmatrix} \tilde{\e}  \\ \tilde{\h} \\ \tilde{\z} \end{pmatrix} \Bigg\rangle_{\tilde{\mathscr{H}}}\\
    \label{MONOT_EQ1}
& = C \left\|   \begin{pmatrix} \tilde{\e}  \\ \tilde{\h} \\ \tilde{\z} \end{pmatrix} \right\|^2_{\tilde{\mathscr{H}}}
    + \int_{G} \big( \CURL \tilde{\e} \cdot \tilde{\h} - \CURL \tilde{\h} \cdot \tilde{\e} \big) \mathrm{d}\mathbf{x} + \xi \int_0^1 \int_{\Gamma} e^{cs} \partial_{s} \tilde{\z} \cdot \tilde{\z} \mathrm{d}\x \mathbf{d}s.
\end{align}

Using Lemma \ref{LEMMA_GREENS_FORMULA} and the boundary condition from Equation (\ref{EQUATION_MAXWELL_BC_1}), we get
\begin{align}
\notag &\int_{G} \big( \CURL \tilde{\e} \cdot \tilde{\h} - \CURL \tilde{\h} \cdot \tilde{\e} \big) \mathbf{d}\mathbf{x} = \int_{\Gamma} \tilde{\h} \times \boldsymbol{\nu} \cdot \tilde{\e}  \d \x \\
\notag & = \int_{\Gamma} \big( \gamma_1 \g(\e' \times \n)\times \n + \gamma_2 \g(\left. \mathbf{Z}' \right|_{s=1})\times \n - \gamma_1 \g(\e \times \n)\times \n - \gamma_2 \g(\zz{1} )\times \n \big)  \cdot \big( \e - \e' \big) \d \x \\
\notag & = \int_{\Gamma} \big( \gamma_1 \g(\e' \times \n) + \gamma_2 \g(\left. \mathbf{Z}' \right|_{s=1}) - \gamma_1 \g(\e \times \n) - \gamma_2 \g(\zz{1} ) \big) \times \n \cdot \big( \e - \e' \big) \d \x \\
\label{MONOT_EQ2} & = \int_{\Gamma} \big( \gamma_1 \g(\e \times \n) + \gamma_2 \g(\zz{1} ) - \gamma_1 \g(\e' \times \n) - \gamma_2 \g(\left. \mathbf{Z}' \right|_{s=1}) \big)  \cdot \big( \e - \e' \big)           \times \n \d \x \\
\notag & = \int_{\Gamma} \gamma_1 \big(  \g(\e \times \n)  -  \g(\e' \times \n) \big)\cdot \big( \e \times \n  - \e' \times \n \big)   + \gamma_2 \big( \g(\zz{1} ) - \g(\left. \mathbf{Z}' \right|_{s=1}) \big)  \cdot \big( \tilde{\mathbf{Z}} |_{s=0} \big) \d \x.
\end{align}
Recalling Assumption \ref{ASSUMPTION_NONLINEARITY_G} and using Cauchy \& Schwarz' inequality, the latter integral can be estimated both on the low
\begin{align}
\label{MONOT_EQ3}
    \begin{split}
        \int_{\Gamma} \gamma_1 \big(  \g(\e \times \n)  -  \g(\e' \times \n) \big)\cdot \big( \e \times \n  &- \e' \times \n \big) \d\x  \geq  \int_{\Gamma} \gamma_1 c_1 \big|  \e \times \n  - \e' \times \n \big|^2 \d\x \\
        &= \gamma_1 c_1 \int_{\Gamma} \tilde{\mathbf{Z}}^2|_{s=0} \d \x = \gamma_1 c_1 \|  \tilde{\mathbf{Z}} |_{s=0}  \|^2_{(L^2(\Gamma))^3}
    \end{split}
\end{align}
and the high side
\begin{align}
\label{MONOT_EQ4}
\begin{split}
\int_{\Gamma} \gamma_2 \big( \g(\zz{1} ) &- \g(\left. \mathbf{Z}' \right|_{s=1}) \big)  \cdot \big( \tilde{\mathbf{Z}} |_{s=0} \big)   \d\x \\
&\geq - \gamma_2 \left( \int_{\Gamma}   \big( \g(\zz{1} ) - \g(\left. \mathbf{Z}' \right|_{s=1}) \big)^2 \d\x \int_{\Gamma}  \tilde{\mathbf{Z}}^2 |_{s=0} \d\x \right)^\frac12 \\
&\geq   - \gamma_2 \left( \int_{\Gamma}   \big( c_2 \big( \zz{1} -  \mathbf{Z}' |_{s=1}) \big)^2 \d\x \int_{\Gamma}   \tilde{\mathbf{Z}}^2 |_{s=0} \d\x \right)^\frac12 \\
&= - \gamma_2 c_2 \|  \tilde{\mathbf{Z}} |_{s=1}  \|_{(L^2(\Gamma))^3} \cdot \|  \tilde{\mathbf{Z}} |_{s=0}  \|_{(L^2(\Gamma))^3}.
\end{split}
\end{align}

Now, consider the latter term in Equation (\ref{MONOT_EQ1}). Integrating by parts, we get
\begin{align}
\notag \xi \int_0^1 \int_{\Gamma} e^{cs} \partial_{s} \tilde{\z} \cdot \tilde{\z} \mathrm{d}\x \mathrm{d}s & = \frac{\xi}{2} \int_{\Gamma} \int_0^1  e^{cs} \partial_{s} (\tilde{\z}^2)   \mathrm{d}s \mathrm{d}\x \\
\notag & = \frac{\xi}{2} \int_{\Gamma} \left( e^{cs}  \tilde{\z}^2 |_{s=0}^{s=1} - \int_0^1 c e^{cs}  \tilde{\z}^2   \mathrm{d}s \right) \mathrm{d}\x \\
\notag & = \frac{\xi}{2} \int_{\Gamma} \left( e^{c}  \tilde{\z}^2 |_{s=1} -  \tilde{\z}^2 |_{s=0}  - \int_0^1 c e^{cs}  \tilde{\z}^2   \mathrm{d}s \right) \mathrm{d}\x \\
\label{MONOT_EQ5} & = \frac{e^{c}\xi}{2} \| \tilde{\z} |_{s=1} \|^2_{(L^2(\Gamma))^3} - \frac{\xi}{2} \| \tilde{\z} |_{s=0} \|^2_{(L^2(\Gamma))^3} - \frac{\xi c}{2} \int_{\Gamma} \int_0^1  e^{cs}  \tilde{\z}^2   \mathrm{d}s  \mathrm{d}\x.
\end{align}
Recalling Equations (\ref{MONOT_EQ2})--(\ref{MONOT_EQ5}), we obtain
\begin{align}
\notag
& C \left\|     \begin{pmatrix} \tilde{\e}  \\ \tilde{\h} \\ \tilde{\z} \end{pmatrix} \right\|^2_{\tilde{\mathscr{H}}}
+ \int_{G} \big( \CURL \tilde{\e} \cdot \tilde{\h} - \CURL \tilde{\h} \cdot \tilde{\e} \big) \mathrm{d}\mathbf{x} + \xi \int_0^1 \int_{\Gamma} e^{cs} \partial_{s} \tilde{\z} \cdot \tilde{\z} \mathrm{d}\x \mathrm{d}s \\
\notag
& \geq C \left\|    \begin{pmatrix} \tilde{\e}  \\ \tilde{\h}   \end{pmatrix} \right\|^2_{\mathcal{H}} + C\xi \tau \int_0^1 \int_{\Gamma} e^{cs} \tilde{\z}^2 \d \x \d s + \big( \gamma_1 c_1 - \frac{\xi}{2}\big)\| \tilde{\z} |_{s=0} \|^2_{(L^2(\Gamma))^3} + \frac{\xi e^{c}}{2} \| \tilde{\z} |_{s=1} \|^2_{(L^2(\Gamma))^3} \\
\notag
&  - \frac{c\xi}{2}\int_0^1 \int_{\Gamma} e^{cs} \tilde{\z}^2 \d \x \d s - \gamma_2 c_2 \|  \tilde{\mathbf{Z}} |_{s=1}  \|_{(L^2(\Gamma))^3} \cdot \|  \tilde{\mathbf{Z}} |_{s=0}  \|_{(L^2(\Gamma))^3}.
\end{align}

Taking now $\xi < 2\gamma_1 c_1$ and applying Cauchy \& Schwarz' inequality, we arrive at
\begin{align}
\notag
 C & \left\|    \begin{pmatrix} \tilde{\e}  \\ \tilde{\h}   \end{pmatrix} \right\|^2_{\mathcal{H}} + C\xi \tau \int_0^1 \int_{\Gamma} e^{cs} \tilde{\z}^2 \d \x \d s + \big( \gamma_1 c_1 - \frac{\xi}{2}\big)\| \tilde{\z} |_{s=0} \|^2_{(L^2(\Gamma))^3} + \frac{\xi e^{c}}{2} \| \tilde{\z} |_{s=1} \|^2_{(L^2(\Gamma))^3} \\
\notag
&-\frac{c\xi}{2}\int_0^1 \int_{\Gamma} e^{cs} \tilde{\z}^2 \d \x \d s - \gamma_2 c_2 \|  \tilde{\mathbf{Z}} |_{s=1}  \|_{(L^2(\Gamma))^3} \cdot \|  \tilde{\mathbf{Z}} |_{s=0}  \|_{(L^2(\Gamma))^3} \\
\notag
&\geq C \left\|  \begin{pmatrix} \tilde{\e}  \\ \tilde{\h}   \end{pmatrix} \right\|^2_{\mathcal{H}} + 2\left( \big( \gamma_1 c_1 - \frac{\xi}{2}\big)\frac{\xi e^c}{2}\right)^{\frac12} \|  \tilde{\mathbf{Z}} |_{s=1}  \|_{(L^2(\Gamma))^3} \cdot \|  \tilde{\mathbf{Z}} |_{s=0}  \|_{(L^2(\Gamma))^3}  \\
\label{MONOT_EQ6}
& + \xi \big(C \tau - \tfrac{c}{2}\big) \int_0^1 \int_{\Gamma} e^{cs} \tilde{\z}^2 \d \x \d s  - \gamma_2 c_2 \|  \tilde{\mathbf{Z}} |_{s=1}  \|_{(L^2(\Gamma))^3} \cdot \|  \tilde{\mathbf{Z}} |_{s=0}  \|_{(L^2(\Gamma))^3}.
\end{align}

Finally, selecting $c$ such that $2\left( \big( \gamma_1 c_1 - \frac{\xi}{2}\big)\frac{\xi e^c}{2}\right)^{\frac12} \geq \gamma_2 c_2$
and then choosing $C > \frac{c}{2\tau}$, the right hand side of Equation (\ref{MONOT_EQ6}) is rendered positive implying the monotonicity of $\mathscr{A}$.

\emph{Maximality}:
By virtue of Browder \& Minty's Theorem \cite[Theorem 2.2]{Ba2010},
it suffices to prove $(C + \lambda) \operatorname{id} + \mathscr{A}$ is surjective for at least one $\lambda > 0$, i.e.,
for any $(\mathbf{F}_1,\mathbf{F}_2,\mathbf{F}_3)^{T} \in \mathscr{H}$, we need to find $(\e, \h, \z)\in D(\mathscr{A})$ such that
\begin{align}
\label{MONOT_EQ7}
\big((C+1) \operatorname{id} + \mathscr{A}\big)\begin{pmatrix} \e  \\ \h \\ \z  \end{pmatrix} = \begin{pmatrix} \mathbf{F}_1  \\ \mathbf{F}_2 \\ \mathbf{F}_3  \end{pmatrix}.
\end{align}

Let $b = C + 1$. From Equation (\ref{MONOT_EQ7}), we have $b\z + \tau^{-1} \partial_{s} \z = \mathbf{F}_3$, whence we easily get
\begin{align}
\label{MONOT_Z_EXACT}
\z(t,s,\mathbf{x}) = e^{-\tau b s} \left( \int_0^s \mathbf{F}_3(t,s,\mathbf{x}) e^{\tau b r} \d r + \e(t,\mathbf{x}) \times \n \right).
\end{align}
In particular,
\begin{align}
\label{MONOT_Z1_EXACT}
\z(t,s,\mathbf{x})|_{s=1} &= e^{-\tau b } \left( \int_0^1 \mathbf{F}_3(t,r,\mathbf{x}) e^{\tau b r} \d r + \e(t,\mathbf{x}) \times \n \right), \\
\label{MONOT_Z0_EXACT}
\z(t,s,\mathbf{x})|_{s=0} &= \e(t,\mathbf{x}) \times \n.
\end{align}
Further, using Equation (\ref{MONOT_EQ7}), we obtain
\begin{equation}
\label{MONOT_H_DEFINE}
\h = b^{-1}\big(\mathbf{F}_2 - \boldsymbol{\mu}^{-1} \CURL \e \big)
\end{equation}
to arrive at
\begin{align}
\label{MONOT_EQ8}
b^2\boldsymbol{\varepsilon} \e - \CURL (\mathbf{F}_2 - \boldsymbol{\mu}^{-1} \CURL \e) = b \boldsymbol{\varepsilon} \mathbf{F}_1.
\end{align}
At the first glance, $\h \in \big(L^{2}(G)\big)^{3}$, but $\CURL \h \in \big(L^{2}(G)\big)^{3}$ and $\boldsymbol{\nu} \times \mathbf{H} \in \big(L^{2}(G)\big)^{3}$
will later follow from the regularity of $\mathbf{E}$ (cf. \cite[p. 38]{NiPi2007}).

Equation (\ref{MONOT_EQ8}) is formally equivalent with
\begin{align}
\label{MONOT_EQ9}
b^2\boldsymbol{\varepsilon} \e + \CURL (\boldsymbol{\mu}^{-1} \CURL \e) = b \boldsymbol{\varepsilon} \mathbf{F}_1 + \CURL \mathbf{F}_2,
\end{align}
while the boundary condition in Equation (\ref{EQUATION_MAXWELL_BC_1}) can formally be transformed to
\begin{align}
\label{MONOT_EQ10}
-b^{-1}\boldsymbol{\mu}^{-1} \CURL \e \times \n + \gamma_1 \g(\zz{0})\times \n + \gamma_2 \g(\zz{1})\times \n = -b^{-1} \mathbf{F}_2 \times \n,
\end{align}
where $\zz{1}$ and $\zz{0}$ are given by Equations (\ref{MONOT_Z1_EXACT}) and (\ref{MONOT_Z0_EXACT}), respectively.

Define the Hilbert space
\begin{equation}
\label{MONOT_SPACE_W}
W_{\boldsymbol{\varepsilon}} = \left\{ \e \in (L^2(G))^3\ |\ \CURL \e \in (L^2(G))^3,\ \DIV(\boldsymbol{\varepsilon} \e) \in L^2(G),\ \e \times \n \in (L^2(\Gamma))^3 \right\}
\end{equation}
endowed with the norm
\begin{equation}
\label{MONOT_NORM_W}
\| \e \|^2_{W_{\boldsymbol{\varepsilon}}} = \int_G |\e|^2 + |\CURL \e|^2 + |\DIV(\boldsymbol{\varepsilon} \e)|^2\ \dx + \int_{\Gamma} |\e \times \n|^2\ \d\x.
\end{equation}

Consider the variational problem: Find $\e \in W_{\boldsymbol{\varepsilon}}$ such that
\begin{equation}
\label{MONOT_VARIATIONAL_PROBLEM}
\a(\e, \e') = \int_G b\boldsymbol{\varepsilon} \mathbf{F}_1 \cdot \e' + \mathbf{F}_2 \cdot \CURL \e'\ \dx \text{ for any } \e' \in W_{\boldsymbol{\varepsilon}}.
\end{equation}
Here, the nonlinear form $\a(\cdot, \cdot)$ is defined by
\begin{align*}
\label{MONOT_FORM_DEFINITION}
\a(\e, \e') :=  & \int_G b^2\boldsymbol{\varepsilon} \e \cdot \e' + \boldsymbol{\mu}^{-1}  \CURL \e \cdot \CURL \e' + s \DIV(\boldsymbol{\varepsilon} \e) \DIV (\boldsymbol{\varepsilon} \e')\ \dx \\
& + b\int_{\Gamma} (\e' \times \n) \cdot \big(\gamma_1 \g(\zz{0}) + \gamma_2 \g (\zz{1})\big)\ \d\x,
\end{align*}
where $\zz{1}$ and $\zz{0}$ are given by Equations (\ref{MONOT_Z1_EXACT}) and (\ref{MONOT_Z0_EXACT}), respectively, and $s$ is a positive number to be chosen later.

Similar to \cite{ElLaNi2002}, consider the operator
\begin{equation}
\opA : W_{\boldsymbol{\varepsilon}} \to W'_{\boldsymbol{\varepsilon}}, \quad \opA u (v) = \a(u,v).
\end{equation}
Observing that right-hand side of Equation (\ref{MONOT_VARIATIONAL_PROBLEM}) belongs to the space $W'_{\boldsymbol{\varepsilon}}$,
the solvability of Equation (\ref{MONOT_VARIATIONAL_PROBLEM}) needs to follow from surjectivity of the operator $\opA$.
Using \cite[Corollary 2.2]{Sho1997} and the fact that strong monotonicity implies coercivity,
it is sufficient to prove $\opA$ is strongly monotone, hemicontinuous and bounded.

\textit{ Strong monotonicity: }
For any $\e,\e' \in W_{\boldsymbol{\varepsilon}}$, letting $\tilde{\e} = \e - \e'$, we have
\begin{align}
\notag
\big\langle \opA \e  - \opA \e', \e &- \e' \big\rangle_{W'_{\boldsymbol{\varepsilon}}\times W_{\boldsymbol{\varepsilon}}} =  \big\langle \opA \e , \e - \e' \big\rangle_{W'_{\boldsymbol{\varepsilon}}\times W_{\boldsymbol{\varepsilon}}} - \big\langle  \opA \e', \e - \e' \big\rangle_{W'_{\boldsymbol{\varepsilon}}\times W_{\boldsymbol{\varepsilon}}} \\
\notag
&= \a(\e,\e - \e') - \a(\e', \e - \e') \\
\notag
&= \int_G b^2\boldsymbol{\varepsilon} \e \cdot \tilde{\e} + \boldsymbol{\mu}^{-1} \CURL \e \cdot \CURL \tilde{\e} + s \DIV(\boldsymbol{\varepsilon} \e)\DIV(\boldsymbol{\varepsilon} \tilde{\e})\ \dx + \\
\notag
&+ b\int_{\Gamma} (\tilde{\e} \times \n) \cdot \Big( \gamma_1 \g (\e \times \n) + \gamma_2 \g\big(e^{-\tau b}(\int_0^1 \mathbf{F}_3(r) e^{\tau b r} \d r + \e \times \n)\big)\Big)\ \d\x \\
\notag
&- \int_G b^2\boldsymbol{\varepsilon} \e' \cdot \tilde{\e} + \boldsymbol{\mu}^{-1} \CURL \e' \cdot \CURL \tilde{\e} + s \DIV(\boldsymbol{\varepsilon} \e')\DIV(\boldsymbol{\varepsilon} \tilde{\e})\ \dx + \\
\notag
&- b\int_{\Gamma} (\tilde{\e} \times \n) \cdot \Big( \gamma_1 \g (\e' \times \n) + \gamma_2 \g\big(e^{-\tau b}(\int_0^1 F_3(r) e^{\tau b r} \d r + \e' \times \n)\big)\Big)\ \d\x \\
\notag
&= \int_G b^2\boldsymbol{\varepsilon} \tilde{\e} \cdot \tilde{\e} + \boldsymbol{\mu}^{-1} \CURL \tilde{\e} \cdot \CURL \tilde{\e} + s \DIV(\boldsymbol{\varepsilon} \tilde{\e})\DIV(\boldsymbol{\varepsilon} \tilde{\e})\ \dx + \\
\notag
&+ b\int_{\Gamma} (\tilde{\e} \times \n) \cdot \Big( \gamma_1 \g (\e \times \n) + \gamma_2 \g\big(e^{-\tau b}(\int_0^1 \mathbf{F}_3(r) e^{\tau b r} \d r + \e \times \n)\big)\Big)\ \d\x \\
\notag
&- b\int_{\Gamma} (\tilde{\e} \times \n) \cdot \Big( \gamma_1
\g (\e' \times \n) + \gamma_2 \g\big(e^{-\tau b}(\int_0^1
\mathbf{F}_3(r) e^{\tau b r} \d r + \e' \times \n)\big)\Big)\ \d\x.
\end{align}
The latter two integrals rewrite as
\begin{align}
\label{MONOT_EQ11}
& b \gamma_1 \int_{\Gamma} (\tilde{\e} \times \n) \cdot \big( \g (\e \times \n) - \g (\e' \times \n)\big) \ \d\x   \\
\notag & + b \gamma_2 \int_{\Gamma} (\tilde{\e} \times \n) \cdot
\Big( \g\big(e^{-\tau b}(\int_0^1 \mathbf{F}_3(r) e^{\tau b r} \d r
+ \e \times \n)\big) -  \g\big(e^{-\tau b}(\int_0^1 \mathbf{F}_3(r)
e^{\tau b r} \d r + \e' \times \n)\big) \Big)\ \d\x.
\end{align}

Utilizing Assumption \ref{ASSUMPTION_NONLINEARITY_G}, we obtain
\begin{align}
b \gamma_1 \int_{\Gamma} (\tilde{\e} \times \n) \cdot \big( \g (\e \times \n) - \g (\e' \times \n)\big) \ \d\x  \geq b \gamma_1  \int_{\Gamma}  c_1 \big| \tilde{\e} \times \n \big|^2 \ \d\x
\end{align}
and
\begin{align}
\notag
& b \gamma_2 \int_{\Gamma} (\tilde{\e} \times \n) \cdot \Big( \g\big(e^{-\tau b}(\int_0^1 \mathbf{F}_3(r) e^{\tau b r} \d r + \e \times \n )\big) -  \g\big(e^{-\tau b}(\int_0^1 \mathbf{F}_3(r) e^{\tau b r} \d r + \e' \times \n)\big) \Big)\ \d\x \\
& \geq b \gamma_2 e^{\tau b} \int_{\Gamma} c_1 \big| e^{-\tau b} \big( \tilde{\e} \times \n \big)\big|^2 \ \d\x = b \gamma_2 c_1 e^{-\tau b} \int_{\Gamma}  \big| \big( \tilde{\e} \times \n \big)\big|^2 \ \d\x.
\end{align}
Hence,
\begin{align}
\notag
\big\langle \opA \e  - \opA \e', \e - \e' \big\rangle_{W'_{\boldsymbol{\varepsilon}}\times W_{\boldsymbol{\varepsilon}}} &\geq \int_G b^2\boldsymbol{\varepsilon} \tilde{\e} \cdot \tilde{\e} + \boldsymbol{\mu}^{-1} | \CURL \tilde{\e} |^2 + s | \DIV(\boldsymbol{\varepsilon} \tilde{\e}) |^2 \ \dx \\
\notag
&+ b c_1 (\gamma_1 + e^{-\tau b}\gamma_2)  \int_{\Gamma}  c_1 \big| \tilde{\e} \times \n \big|^2 \ \d\x \\
\notag
&\geq c^{*} \| \e - \e' \|^2_{W_{\boldsymbol{\varepsilon}}}
\end{align}
for some positive $c^{\ast}$.

\textit{Hemicontinuity:} For any $\e,\e'\in W_{\boldsymbol{\varepsilon}}$, we can write
\begin{align}
\label{MONOT_EQ12}
\big\langle \opA (\e &+ t\e') , \e' \big\rangle_{W'_{\boldsymbol{\varepsilon}}\times W_{\boldsymbol{\varepsilon}}} = \ \a(\e + t \e' , \e') \\
\notag
&= \int_{G} b^2 \boldsymbol{\varepsilon}(\e + t \e') \cdot \e' + \boldsymbol{\mu}^{-1} \CURL(\e + t \e') \cdot \CURL \e' + s \DIV\big(\boldsymbol{\varepsilon}(\e + t \e')\big)\DIV(\boldsymbol{\varepsilon} \e')\ \dx \\
\notag
& + b\gamma_1\int_{\Gamma} (\e' \times \n) \cdot   \g \big((\e+t\e')\times\n \big) \d \x \\
\notag & + b \gamma_2\int_{\Gamma} (\e' \times \n) \cdot  \g \Big(
e^{-\tau b } ( \int_0^1 \mathbf{F}_3(r) e^{\tau b r} \d r + (\e +
t\e') \times \n ) \Big) \d\x.
\end{align}
On the strength of Assumption \ref{ASSUMPTION_NONLINEARITY_G}, we get the continuity of $\g(\cdot)$.
Now, by virtue of Equation (\ref{MONOT_EQ12}), the continuity of $t \mapsto \big\langle \opA (\e+t\e') , \e' \big\rangle_{W'_{\boldsymbol{\varepsilon}}\times W_{\boldsymbol{\varepsilon}}}$ follows.

\textit{Boundedness:} Suppose $\|\e\|_{W_{\boldsymbol{\varepsilon}}} \leq c$. Then,
\begin{align}
\notag
\left| \big\langle \opA \e , \e' \big\rangle_{W'_{\boldsymbol{\varepsilon}}\times W_{\boldsymbol{\varepsilon}}} \right| &= \left|\a(\e , \e')\right| \\
\notag
&\leq \int_{G} b^2 \left| \boldsymbol{\varepsilon} \e \cdot \e' \right|+  \left| \boldsymbol{\mu}^{-1} \CURL \e  \cdot \CURL \e' \right| + s \left| \DIV(\boldsymbol{\varepsilon} \e )\DIV(\boldsymbol{\varepsilon} \e')\right|\ \dx \\
\notag
&+ b\gamma_1 \int_{\Gamma} \left| (\e' \times \n) \cdot  \g \big( \e \times \n \big)\right| \d \x \\
\notag
&+ b \gamma_2\int_{\Gamma} \left| (\e' \times \n) \cdot  \g
\big( e^{-\tau b } ( \int_0^1 \mathbf{F}_3(r) e^{\tau b r} \d r + \e
\times \n ) \big) \right|\d\x.
\end{align}

Using Cauchy \& Schwarz' inequality and Assumption \ref{ASSUMPTION_NONLINEARITY_G}, we estimate
\begin{align}
\notag
  b\gamma_1 \int_{\Gamma} \big| (\e' \times \n) &\cdot \g \big( \e \times \n \big)\big| \d \x \leq  b\gamma_1  \left( \int_{\Gamma} |\e'\times \n|^2 \d\x \int_{\Gamma} |\g(\e \times \n)|^2 \d\x \right)^{1/2} \\
\notag
& \leq  b\gamma_1 c_2 \left( \int_{\Gamma} |\e'\times \n|^2 \d\x \int_{\Gamma} |\e \times \n|^2 \d\x \right)^{1/2} \\
\notag
& =  b\gamma_1 c_2  \|\e'\times \n\|_{(L^2(\Gamma))^3} \|\e \times \n\|_{(L^2(\Gamma))^3} \\
\notag & \leq b\gamma_1 c_2  \|\e'\|_{W_{\boldsymbol{\varepsilon}}} \|\e
\|_{W_{\boldsymbol{\varepsilon}}} \\
\notag & \leq b\gamma_1 c_2 c  \|\e'\|_{W_{\boldsymbol{\varepsilon}}}
\end{align}
and
\begin{align}
\notag b \gamma_2\int_{\Gamma} & \left| (\e' \times  \n)  \cdot \g
\big( e^{-\tau b } ( \int_0^1 \mathbf{F}_3(r) e^{\tau b r} \d
r + \e
\times \n ) \big) \right|\d\x \\
\notag & \leq b \gamma_2 c_2 \left( \int_{\Gamma} \left| \e' \times
\n \right|^2 \d\x  \int_{\Gamma} \left| e^{-\tau b } \big(
\int_0^1 \mathbf{F}_3(r) e^{\tau b r} \d r + \e \times \n \big)  \right|^2
\d\x \right)^{1/2} \\
\notag & \leq  b \gamma_2 c_2 e^{-\tau b } \|\e'\times
\n\|_{(L^2(\Gamma))^3}  \left(  \int_{\Gamma} 2 \big( \int_0^1
\mathbf{F}_3(r) e^{\tau b r} \d r \big)^2+ 2\left| \e \times \n \right|^2
\d\x \right)^{1/2} \\
\notag & \leq 2 b \gamma_2 c_2 e^{-\tau b }
\|\e'\|_{W_{\boldsymbol{\varepsilon}}} \left(  \|\mathcal{I} \mathbf{F}_3 \|_{(L^2(\Gamma))^3} +
\|\e\times \n\|_{(L^2(\Gamma))^3} \right) \\
\notag & \leq 2 b \gamma_2 c_2 e^{-\tau b }
\|\e'\|_{W_{\boldsymbol{\varepsilon}}} \left(  \|\mathcal{I} \mathbf{F}_3 \|_{(L^2(\Gamma))^3} + c
\right),
\end{align}
where $\mathcal{I} \mathbf{F}_3 = \int_0^1 \mathbf{F}_3(r) e^{\tau b r} \d r$.
Therefore, $\left| \big\langle \opA \e , \e' \big\rangle_{W'_{\boldsymbol{\varepsilon}}\times
W_{\boldsymbol{\varepsilon}}} \right| \leq c^* \|\e'\|_{W_{\boldsymbol{\varepsilon}}}$ for a suitable $c^{\ast}$.
Thus, $\| \opA \e\|_{W'_{\boldsymbol{\varepsilon}}} \leq c^* $ and the conclusion follows.

In summary, $\opA$ is surjective and the problem (\ref{MONOT_VARIATIONAL_PROBLEM}) possesses a (weak) solution.
Since $\opA$ is strongly monotone, the solution is unique.

\emph{Strongness of solution: }
We now prove the (weak) solution $\e \in W_{\boldsymbol{\varepsilon}}$ to
Equation (\ref{MONOT_VARIATIONAL_PROBLEM}) along with corresponding $\h,\z$ satisfy Equation (\ref{MONOT_EQ7}).

First, we show that $\DIV(\boldsymbol{\varepsilon} \e) = 0$. Following \cite{ElLaNi2002}, consider the set
\begin{align}
    D = \{ \varphi \in H_0^1(G)\ | \ \DIV(\boldsymbol{\varepsilon} \nabla \varphi) \in L^2(G)\}.
\end{align}
Letting $\e' = \nabla \varphi$ for arbitrary, but fixed $\varphi \in D$, we can rewrite Equation (\ref{MONOT_VARIATIONAL_PROBLEM}) as
\begin{align}
\notag
\int_G b^2\boldsymbol{\varepsilon} \e \cdot  \nabla \varphi +  s \DIV(\boldsymbol{\varepsilon} \e) \DIV (\boldsymbol{\varepsilon}  \nabla \varphi)\ \dx
= \int_G b\boldsymbol{\varepsilon} \mathbf{F}_1 \cdot \nabla \varphi\ \dx \text{ for any } \varphi \in D.
\end{align}
Using Green's formula, we get
\begin{align}
\label{MONOT_EQ13}
& \int_G -b^2 \DIV(\boldsymbol{\varepsilon} \e) \varphi +  s \DIV(\boldsymbol{\varepsilon} \e) \DIV (\boldsymbol{\varepsilon}  \nabla \varphi)\ \dx  = -\int_G b\DIV(\boldsymbol{\varepsilon} \mathbf{F}_1) \varphi\ \dx \text{ for any } \varphi \in D.
\end{align}
Since $(\mathbf{F}_1,\mathbf{F}_2,\mathbf{F}_3)^T \in \mathscr{H}$, it follows that $\mathbf{F}_1 \in H(\DIV_{\boldsymbol{\varepsilon}} 0, G)$. Thus, the latter integral in Equation (\ref{MONOT_EQ13}) vanishes.
Hence,
\begin{align}
\label{MONOT_EQ14}
& \int_G \DIV(\boldsymbol{\varepsilon} \e) \big( -b^2  \varphi +  s \DIV (\boldsymbol{\varepsilon}  \nabla \varphi) \big)\ \dx  = 0 \text{ for any } \varphi \in D.
\end{align}

Since the spectrum of $\DIV (\boldsymbol{\varepsilon} \nabla \cdot)$ with homogeneous Dirichlet boundary conditions is discrete,
there exists a positive number $s$ such that $b^2/s$ belongs to the resolvent set.
Then, from Equation (\ref{MONOT_EQ14}), we conclude that $\DIV(\boldsymbol{\varepsilon} \e) = 0$ holds strongly in $G$.

Therefore, Equation (\ref{MONOT_VARIATIONAL_PROBLEM}) becomes
\begin{align}
\notag
\int_G b^2\boldsymbol{\varepsilon} \e \cdot \e' &+ \boldsymbol{\mu}^{-1}  \CURL \e \cdot \CURL \e'  + b\int_{\Gamma} (\e' \times \n) \cdot \big(\gamma_1 \g(\zz{0}) + \gamma_2 \g (\zz{1})\big)\ \d\x \\
\label{MONOT_EQ15}
&= \int_G b\boldsymbol{\varepsilon} \mathbf{F}_1 \cdot \e' + \mathbf{F}_2 \cdot \CURL \e'\ \dx \text{ for any } \e' \in W_{\boldsymbol{\varepsilon}}.
\end{align}

Recalling the definition of $\h$ from Equation (\ref{MONOT_H_DEFINE}) and applying Green's formula to Equation (\ref{MONOT_EQ15}), we arrive at
\begin{align}
\notag
\int_G \boldsymbol{\varepsilon} b \e \cdot \e' + \h \cdot \CURL \e'\ \dx &+ \int_{\Gamma} (\e' \times \n) \cdot \big(\gamma_1 \g(\zz{0}) + \gamma_2 \g (\zz{1})\big)\ \d\x \\
\notag
&= \int_G \boldsymbol{\varepsilon} \mathbf{F}_1 \cdot \e'\ \dx \text{ for any } \e' \in W_{\boldsymbol{\varepsilon}}.
\end{align}
Choosing $\e' = P_{\boldsymbol{\varepsilon}} \boldsymbol{\chi}$ with $\boldsymbol{\chi}\in (\mathcal{D}(G))^3$, we get
\begin{align}
\notag
\int_G \boldsymbol{\varepsilon} b \e \cdot  P_{\boldsymbol{\varepsilon}} \boldsymbol{\chi} + \h \cdot \CURL ( P_{\boldsymbol{\varepsilon}} \boldsymbol{\chi})\ \dx   = \int_G \boldsymbol{\varepsilon} \mathbf{F}_1 \cdot  P_{\boldsymbol{\varepsilon}} \boldsymbol{\chi}\ \dx \text{ for any } \boldsymbol{\chi}\in (\mathcal{D}(G))^3
\end{align}
or, after using Green's formula,
\begin{align}
\notag
\int_G \big( \boldsymbol{\varepsilon} b \e - \CURL \h  \big)\cdot  P_{\boldsymbol{\varepsilon}} \boldsymbol{\chi}\  \dx   = \int_G \boldsymbol{\varepsilon} \mathbf{F}_1 \cdot  P_{\boldsymbol{\varepsilon}} \boldsymbol{\chi}\ \dx \text{ for any } \boldsymbol{\chi}\in (\mathcal{D}(G))^3.
\end{align}
Since $P_{\boldsymbol{\varepsilon}} (\mathcal{D}(G))^3$ is dense in $H(\DIV_{\boldsymbol{\varepsilon}} 0, G)$, there identity
\begin{equation}
\label{MONOT_EQ16}
\boldsymbol{\varepsilon} b \e - \CURL \h = \boldsymbol{\varepsilon} \mathbf{F}_1
\end{equation}
follows in the strong sense.

Choosing $\e' = P_{\boldsymbol{\varepsilon}} \boldsymbol{\chi}$ with $\boldsymbol{\chi}\in (C^{\infty}(G))^3$, we get
\begin{align}
\notag
\int_G \boldsymbol{\varepsilon} b \e \cdot  P_{\boldsymbol{\varepsilon}} \boldsymbol{\chi} &+ \h \cdot \CURL ( P_{\boldsymbol{\varepsilon}} \boldsymbol{\chi})\ \dx + \int_{\Gamma} (P_{\boldsymbol{\varepsilon}} \boldsymbol{\chi} \times \n) \cdot \big(\gamma_1 \g(\zz{0}) + \gamma_2 \g (\zz{1})\big)\ \d\x \\
\notag
&= \int_G \boldsymbol{\varepsilon} \mathbf{F}_1 \cdot  P_{\boldsymbol{\varepsilon}} \boldsymbol{\chi}\ \dx \text{ for all } \boldsymbol{\chi}\in (\mathcal{D}(G))^3.
\end{align}
Using Equation (\ref{MONOT_EQ16}), Lemma \ref{REMARK_PROPERTIES_OF_P} and Green's formula, we finally conclude
\begin{align}
\notag
\int_{\Gamma} -\big(\h \times \n\big) \cdot \boldsymbol{\chi}\ \d\x + \int_{\Gamma} \Big(\n \times \big(\gamma_1 \g(\zz{0}) + \gamma_2 \g (\zz{1})\big)\Big) \cdot \boldsymbol{\chi} \, \d\x  = 0
\end{align}
for all $\boldsymbol{\chi}\in (\mathcal{D}(G))^3$.
Thus, we have $ - \h \times \n -  \big(\gamma_1 \g(\zz{0}) + \gamma_2 \g (\zz{1})\big) \times  \n = 0$ in the strong sense.
Therefore, $(\e,\h,\z)^T \in D(\mathscr{A})$ and Equation (\ref{MONOT_EQ7}) is satisfied.
\end{proof}

\begin{theorem}
    \label{THEOREM_WELL_POSEDNESS}

    Under Assumptions \ref{ASSUMPTION_BASIC_CONDITIONS_COEFFICIENTS} and \ref{ASSUMPTION_NONLINEARITY_G}, suppose $\mathbf{V}^{0} \in \mathscr{H}$.
    Then, Equation (\ref{EQUATION_ABSTRACT_CAUCHY_PROBLEM}) possesses a unique global mild solution
    \begin{equation}
        \mathbf{V} \in C^{0}\big([0, \infty), \mathscr{H}\big). \notag
    \end{equation}
    If, moreover, $\mathbf{V}^{0} \in D(\mathscr{A})$, the mild solution $\mathbf{V}$ is a strong solution satisfying
    \begin{equation}
        \mathbf{V} \in W^{1, \infty}_{\mathrm{loc}}(0, \infty; \mathscr{H}) \cap
        L^{\infty}_{\mathrm{loc}}\big(0, \infty; D(\mathscr{A})\big). \notag
    \end{equation}
\end{theorem}

\begin{proof}
    Since the operator $C \operatorname{id} + \mathcal{A}$ is maximally monotone for a sufficiently large $C > 0$,
    using \cite[Corollary 4.1]{Ba2010}, any initial value $\mathbf{V}^{0} \in \overline{D(\mathscr{A})}$ admits a unique mild solution.
    By virtue of Lemma \ref{LEMMA_DENSITY}, this remains true for $\mathbf{V}^{0} \in \mathscr{H}$.
    As for the strong solution, \cite[Theorem 4.5]{Ba2010} applies.
\end{proof}

\begin{remark}
    In contrast to Datko's `counterexamples' of destabilizing boundary delays,
    our nonlinear system (\ref{EQUATION_MAXWELL_1})--(\ref{EQUATION_MAXWELL_IC_HISTORY}) as well as its linearization studied by Nicaise \& Pignotti \cite{NiPi2007}
    are well-posed for two basic reasons:
    1) the boundary conditions involve instanteneous terms of matching order and has a correct sign;
    2) the orders of the delayed and the instantaneous terms are not too high.
    Indeed, adopting the step method commonly used for difference-differential equations, 
    a delayed system can only be well-posed if the delay operator constitutes an `admissible control operator'
    as widely applied in infinite-dimensional control theory.
    In this sense, systems with maximal $L^{p}$-regularity (which Equations (\ref{EQUATION_MAXWELL_1})--(\ref{EQUATION_MAXWELL_IC_HISTORY}) are lacking)
    subject to `strong' delay can behave completely differently from those without this important property.
\end{remark}

\section{Exponential Stability}
\label{SECTION_EXPONENTIAL_STABILITY}

Our thrust is to prove the exponential stability for Equations (\ref{EQUATION_MAXWELL_1})--(\ref{EQUATION_MAXWELL_IC_HISTORY}).
To this end, we consider the ``natural energy'' functional
\begin{align}
    E(t) := \tfrac{1}{2} \big\|\mathbf{V}\big\|_{\mathscr{H}}^{2} \equiv
    \frac{1}{2} \int_{G} \big|\mathbf{E}(t, \cdot)\big|^{2} \mathrm{d}\mathbf{x} +
    \frac{1}{2} \int_{G} \big|\mathbf{H}(t, \cdot)\big|^{2} \mathrm{d}\mathbf{x} +
    \tau \int_{0}^{1} \int_{\Gamma} \big|\mathbf{E}(t - \tau s, \mathbf{x}) \times \boldsymbol{\nu}\big|^{2}
    \mathrm{d}\mathbf{x} \mathrm{d}s. \notag
\end{align}
In the following, we apply a combination of Rellich's multiplier techniques developed for boundary control problems
along with Lyapunov's techniques for delay differential equations in the spirit of \cite{KhuPoRa2015}.

For $\mathbf{x}_{0} \in \mathbb{R}^{3}$, consider the vector field $\mathbf{m}(\mathbf{x}) := \mathbf{x} - \mathbf{x}_{0}$.
\begin{assumption}[Regularity and geometric conditions]
    \label{ASSUMPTION_GEOMETRY}

    Suppose the following conditions are satisfied:
    \begin{enumerate}
        \item $G$ is a bounded $C^{2}$-domain.

        \item $G$ is strictly star-shaped with respect to $\mathbf{x}_{0} \in G$, i.e.,
        \begin{equation}
            \mathbf{m}(\mathbf{x}) \cdot \boldsymbol{\nu}(\mathbf{x}) > 0 \text{ for } \mathbf{x} \in \Gamma.
            \label{EQUATION_GEOMETRIC_CONDITION}
        \end{equation}

        \item $\boldsymbol{\varepsilon}, \boldsymbol{\mu} \in C^{1}(\bar{G}, \mathbb{R}^{3 \times 3})$.

        \item There exists a constant $d_1 > 0$ such that
        \begin{equation}
         \label{STABILITY_ESTIMATIONS_FOR_M}
                \boldsymbol{\varepsilon} + (\mathbf{m} \cdot \nabla) \boldsymbol{\varepsilon} \geq d_{1} \boldsymbol{\varepsilon} \quad \text{ and } \quad
                \boldsymbol{\mu} + (\mathbf{m} \cdot \nabla) \boldsymbol{\mu} \geq d_{1} \boldsymbol{\mu} \text{ in } \bar{G},
        \end{equation}
    \end{enumerate}
\end{assumption}

\begin{remark}
    Inequalities (\ref{STABILITY_ESTIMATIONS_FOR_M}) are mathematical assumptions on the physical nature of the medium (cf. \cite{El2007})
    and the geometry of the domain $G$ -- the latter inasmuch as the function $\mathbf{m}(\cdot)$ is involved.
    Similar conditions are imposed in \cite{ElMa2002, Kri2007}, etc.
    In case both $\boldsymbol{\varepsilon}$ and $\boldsymbol{\mu}$ are scalar and constant (or ``nearly'' constant),
    this corresponds to the ``strict star-shapedness'' with respect to $\mathbf{x}_{0}$ (see, e.g., \cite[p. 48]{Ko1994}).
    In particular, all convex domains are strictly star-shaped.
    Hence, the geometry class is non-trivial.
\end{remark}

Consider a new functional
\begin{align}
\label{STABILITY_XI_ENERGY}
    E_\xi(t) :=     \frac{1}{2} \int_{G} \big|\mathbf{E}(t, \cdot)\big|^{2} \mathrm{d}\mathbf{x} +
    \frac{1}{2} \int_{G} \big|\mathbf{H}(t, \cdot)\big|^{2} \mathrm{d}\mathbf{x} +
    \xi \tau \int_{0}^{1} \int_{\Gamma} \big|\mathbf{E}(t - \tau s, \mathbf{x}) \times \boldsymbol{\nu}\big|^{2}
    \mathrm{d}\mathbf{x} \mathrm{d}s, \notag
\end{align}
where $\xi$ is a positive number such that
\begin{equation}
\label{STABILITY_CONDITION_ON_XI}
\gamma_{1} c_1 - \frac{\gamma_2  c_2 }{2} > \xi > \frac{\gamma_2 c_2}{2}.
\end{equation}
Obviously, $\xi$ exists if $\gamma_1 c_1 > \gamma_2 c_2$.

\begin{lemma}
\label{STABILITY_ENERGY_DUMPING_INEQUALITY}
Suppose $\gamma_1 c_1 > \gamma_2 c_2$. Then, there exist positive numbers $c_1^E, c_2^E$ such that for all $t_2 > t_1 \geq 0$ the following inequality holds
\begin{align}
- c_1^E \int_{t_1}^{t_2} \int_{\Gamma} \zz{0}^2 + \zz{1}^2 \d\x \d t \ge E_{\xi}(t_2) - E_{\xi}(t_1) \ge - c_2^E \int_{t_1}^{t_2} \int_{\Gamma} \zz{0}^2 + \zz{1}^2 \d\x \d t,
\end{align}
where $(\e,\h,\z)^T$ is a strong solution of Equation (\ref{EQUATION_ABSTRACT_CAUCHY_PROBLEM}).
\end{lemma}

\begin{proof}
Similar to \cite[Lemma 2.7]{ElLaNi2002}, multiplying Equations
(\ref{EQUATION_MAXWELL_1}) and (\ref{EQUATION_MAXWELL_2}) in
$L^{2}\big(0, T; (L^{2}(G))^{3}\big)$ with $\mathbf{E}$ and
$\mathbf{H}$, respectively, integrating by parts and using the
boundary condition from Equation (\ref{EQUATION_MAXWELL_BC_1}), we
get
\begin{align}
\label{STABILITY_EQ_1}
E_{\xi}(t_2) - E_{\xi}(t_1) & = - \int_{t_1}^{t_2} \int_{\Gamma} \big( \gamma_{1} \g( \zz{0} ) + \gamma_2 \g(\zz{1}) \big) \cdot  \big(\e(t, \cdot) \times \n\big) \dx \, \d t \\
\notag
& + \xi \tau \int_{t_1}^{t_2} \int_0^1 \int_{\Gamma} 2 \big( \e(t-\tau s ,\cdot) \times \n \big) \cdot \partial_{t} \big( \e(t - \tau s,\cdot) \times \n \big)   \dx \d s \d t.
\end{align}
Recalling
\begin{equation}
\mathbf{Z}(t, s, \cdot) = \mathbf{E}(t - \tau s, \cdot) \times \boldsymbol{\nu} \text{ for } s \in [0, 1] \notag
\end{equation}
and following \cite{KhuPoRa2015}, we obtain
\begin{equation}
\tau \partial_{t} \mathbf{Z}(t, s, \cdot) + \partial_{s} \mathbf{Z}(t, s, \cdot) = \mathbf{0} \text{ for } (t, s) \in (0, \infty) \times (0, 1). \notag
\end{equation}
Therefore,
\begin{align}
\notag
 \xi \tau \int_0^1 \int_{\Gamma} 2 \big( \e(t-\tau s ,\cdot) \times \n \big) & \cdot \partial_{t} \big( \e(t - \tau s,\cdot) \times \n \big)   \d\x \d s \\
\notag
& =  - \xi \int_0^1 \int_{\Gamma} 2 \big( \e(t-\tau s ,\cdot) \times \n \big) \cdot  \partial_{s} \big( \e(t - \tau s,\cdot) \times \n \big)   \d\x \d s \\
\notag
& =  - \xi \int_0^1 \int_{\Gamma}   \partial_{s} \big| \e(t - \tau s,\cdot) \times \n \big|^2   \d\x \d s \\
\notag
& =  - \xi  \int_{\Gamma} \left.   \big| \e(t - \tau s,\cdot) \times \n \big|^2 \right|_{s=0}^{s=1}  \d\x \\
\notag
& =   \xi  \int_{\Gamma} \zz{0}^2 - \zz{1}^2  \d\x.
\end{align}
After plugging the latter identity into Equation (\ref{STABILITY_EQ_1}), we arrive at
\begin{align}
\label{STABILITY_EQ_2}
E_{\xi}(t_2) - E_{\xi}(t_1) & = - \int_{t_1}^{t_2} \int_{\Gamma} \big( \gamma_{1} \g( \zz{0} ) + \gamma_2 \g(\zz{1}) \big) \cdot  \zz{0} \dx \, \d t \\
\notag
& + \xi  \int_{t_1}^{t_2}  \int_{\Gamma} \zz{0}^2 - \zz{1}^2   \d\x  \d t.
\end{align}

Using Assumption \ref{ASSUMPTION_NONLINEARITY_G} and Young's inequality, we get
\begin{align}
\label{STABILITY_EQ_3}
\int_{\Gamma} \g(\zz{0}) \cdot \zz{0} \d\x & \geq c_1 \int_{\Gamma} \zz{0}^2 \d\x \text{ and } \\
\notag
\int_{\Gamma} \g(\zz{1}) \cdot \zz{0} \d\x & \geq  -c_2\int_{\Gamma} \left| \zz{1} \cdot \zz{0} \right| \d\x  \\
\label{STABILITY_EQ_4}
& \geq -\frac{c_2}{2}  \int_{\Gamma} \zz{1}^2 \d\x - \frac{c_2}{2} \int_{\Gamma} \zz{0}^2 \d\x.
\end{align}
Then, Equation (\ref{STABILITY_EQ_2}) can further be estimated as follows:
\begin{align}
\notag
E_{\xi}(t_2) - E_{\xi}(t_1)  \leq &  \int_{t_1}^{t_2} \left(  -\gamma_{1} c_1\int_{\Gamma}   \zz{0}^2 \d\x + \frac{\gamma_2 c_2}{2} \int_{\Gamma} \zz{1}^2 \d\x + \frac{\gamma_2 c_2}{2} \int_{\Gamma}  \zz{0}^2 \d \x \right) \d t \\
\notag
&+ \xi  \int_{t_1}^{t_2}  \int_{\Gamma} \zz{0}^2 - \zz{1}^2   \d\x  \d t \\
\label{STABILITY_EQ_5}
&= -  (\gamma_{1} c_1 - \frac{\gamma_2  c_2 }{2} -\xi) \int_{t_1}^{t_2} \int_{\Gamma}   \zz{0}^2 \d\x \d t - (\xi-\frac{\gamma_2 c_2}{2}) \int_{t_1}^{t_2} \int_{\Gamma} \zz{1}^2 \d\x   \d t.
\end{align}
Since $\xi$ is selected to satisfy Equation (\ref{STABILITY_CONDITION_ON_XI}), we arrive at
\begin{align}
\notag
E_{\xi}(t_2) - E_{\xi}(t_1) & \leq - c_1^E  \int_{t_1}^{t_2} \int_{\Gamma}\zz{0}^2 + \zz{1}^2 \d\x \d t.
\end{align}

On the other hand,
\begin{align}
\label{STABILITY_EQ_6}
\left| \int_{\Gamma} \g(\zz{0}) \cdot \zz{0} \d\x \right| & \leq c_2 \int_{\Gamma} \zz{0}^2 \d\x \text{ and } \\
\notag
\left| \int_{\Gamma} \g(\zz{1}) \cdot \zz{0} \d\x \right| & \leq  c_2\int_{\Gamma} \left| \zz{1} \cdot \zz{0} \right| \d\x  \\
\label{STABILITY_EQ_7}
&\leq \frac{c_2}{2}  \int_{\Gamma} \zz{1}^2 \d\x + \frac{c_2}{2} \int_{\Gamma} \zz{0}^2 \d\x.
\end{align}
Thus,
\begin{align}
\notag
E_{\xi}(t_2) - E_{\xi}(t_1) &\geq - \left| \int_{t_1}^{t_2} \int_{\Gamma} \big( \gamma_{1} \g( \zz{0} ) + \gamma_2 \g(\zz{1}) \big) \cdot  \zz{0}  \dx \, \d t \right|  \\
\notag
&- \xi \left| \int_{t_1}^{t_2}  \int_{\Gamma} \zz{0}^2 - \zz{1}^2   \d\x  \d t \right| \\
\notag
&\geq \int_{t_1}^{t_2} \left(-  \gamma_{1} c_2\int_{\Gamma}   \zz{0}^2 \d\x - \frac{\gamma_2 c_2}{2} \int_{\Gamma} \zz{1}^2 \d\x - \frac{\gamma_2 c_2}{2} \int_{\Gamma}  \zz{0}^2 \d \x \right) \d t \\
\notag
&- \xi  \int_{t_1}^{t_2}  \int_{\Gamma} \zz{0}^2  \d\x  \d t - \xi  \int_{t_1}^{t_2}  \int_{\Gamma} \zz{1}^2   \d\x  \d t \\
\notag
&\geq -  c_2^E \int_{t_1}^{t_2} \int_{\Gamma}   \zz{0}^2 + \zz{1}^2 \d\x \d t,
\end{align}
which finishes the proof.
\end{proof}

\begin{lemma}
\label{STABILITY_LEMMA_2}
There exist positive numbers $c, c_T$ such that the estimate
\begin{equation}
\int_{0}^{T} E_{\xi} (t)\ \d t \le c \big(E_{\xi}(0) +
E_{\xi}(T)\big) + c_T \int_{0}^{T} \int_{\Gamma} \zz{0}^2 +
\zz{1}^2\ \d\x \d t
\end{equation}
holds true for every $T > 0$ along any strong solution $(\e,\h,\z)^T$ of Equation (\ref{EQUATION_ABSTRACT_CAUCHY_PROBLEM}).
\end{lemma}

\begin{proof}
Similar to \cite[Section 3.1, pp. 193--195]{El2007}, using Rellich's multipliers
$\mathbf{m} \times (\boldsymbol{\varepsilon} \mathbf{E})$ and $\mathbf{m} \times (\boldsymbol{\mu} \mathbf{H})$, we obtain
\begin{align}
\label{STABILITY_EQ_15}
\begin{split}
\frac{1}{2} \int_{0}^{T} \int_{G} &\big( \boldsymbol{\varepsilon} + (\m \cdot \nabla ) \boldsymbol{\varepsilon} \big)\e \cdot \e +  \big( \boldsymbol{\mu} + (\m \cdot \nabla ) \boldsymbol{\mu} \big)\h \cdot \h  \dx \d t \\
&= -\frac{1}{2}\int_0^T \int_\Gamma \n \cdot \m \big( \boldsymbol{\mu}\h \cdot \h + \boldsymbol{\varepsilon} \e \cdot \e \big) \d\x\d t \\
&+ \int_0^T \int_\Gamma \big( \n \times \e \big)\cdot \big( \m \times \boldsymbol{\varepsilon} \e \big) \ \d\x\d t + \int_0^T \int_\Gamma \big( \n \times \h \big)\cdot \big( \m \times \boldsymbol{\mu} \h \big) \ \d\x\d t \\
&+ \int_G \big( \m \times \boldsymbol{\varepsilon} \e(T) \big) \cdot \big(\boldsymbol{\mu} \h(T) \big)\dx -  \int_G \big( \m \times \boldsymbol{\varepsilon} \e(0) \big) \cdot \big(\boldsymbol{\mu} \h(0) \big)\dx.
\end{split}
\end{align}
The left-hand side can be estimated using inequalities in Equation (\ref{STABILITY_ESTIMATIONS_FOR_M}) as
\begin{align}
\notag
\frac{1}{2} \int_{0}^{T} \int_{G} &\big( \boldsymbol{\varepsilon} + (\m \cdot \nabla ) \boldsymbol{\varepsilon} \big)\e \cdot \e +  \big( \boldsymbol{\mu} + (\m \cdot \nabla ) \boldsymbol{\mu} \big)\h \cdot \h \ \dx \d t  \geq \frac{d_1}{2} \int_{0}^{T} \int_{G} \boldsymbol{\varepsilon} \e \cdot \e + \boldsymbol{\mu} \h \cdot \h \d\x \d t.
\end{align}
From Assumption \ref{ASSUMPTION_BASIC_CONDITIONS_COEFFICIENTS} and Equation (\ref{ASSUMPTIONS_ALPHA}),
we get $\boldsymbol{\varepsilon} \e \cdot \e \ge \alpha |\e|^2,\ \boldsymbol{\mu} \e \cdot \e \ge \alpha |\e|^2$ for all $E\in \mathbb{R}^3$. Therefore,
\begin{align}
\label{STABILITY_EQ_8}
\begin{split}
\frac{1}{2} \int_{0}^{T} \int_{G} \big( \boldsymbol{\varepsilon} + (\m \cdot \nabla ) \boldsymbol{\varepsilon} \big)\e \cdot \e &+ \big( \boldsymbol{\mu} + (\m \cdot \nabla ) \boldsymbol{\mu} \big)\h \cdot \h\ \dx \d t \\
&\geq \frac{d_1 \alpha}{2} \int_{0}^{T} \int_{G}\left| \e \right|^2 + \left| \h \right|^2 \d\x \d t.
\end{split}
\end{align}

From the compactness of $\Gamma$ and the continuity of $\m$, we get $\m \cdot \n \geq \beta > 0$ uniformly on $\Gamma$.
Thus, the first term on the right-hand side of Equation (\ref{STABILITY_EQ_15}) can be estimated via
\begin{align}
\label{STABILITY_EQ_14}
\begin{split}
-\frac{1}{2}\int_0^T \int_\Gamma \n \cdot \m \big( \boldsymbol{\mu}\h \cdot \h + \boldsymbol{\varepsilon} \e \cdot \e \big) \d\x\d t &\leq -\frac{1}{2}\int_0^T \int_\Gamma \n \cdot \m \big( \boldsymbol{\mu}\h \cdot \h+ \boldsymbol{\varepsilon} \e \cdot \e \big) \d\x\d t  \\
&\leq -\frac{ \beta }{2}\int_0^T \int_\Gamma \boldsymbol{\mu}\h \cdot \h +  \boldsymbol{\varepsilon} \e \cdot \e  \d\x\d t.
\end{split}
\end{align}

Utilizing Young's inequality, we further get
\begin{align}
\label{STABILITY_EQ_9}
\begin{split}
\Big| \int_G \big( \m \times \boldsymbol{\varepsilon} \e(T) \big) &\cdot \big(\boldsymbol{\mu} \h(T) \big)\dx \Big| \leq \int_G | \m | \cdot \left| \boldsymbol{\varepsilon} \e(T) \right| \cdot \left| \boldsymbol{\mu} \h(T) \right| \dx \\
& \leq \sup_{\x\in G} |\m(\x)| \left|\lambda_{\max}(\boldsymbol{\varepsilon}) \lambda_{\max}(\boldsymbol{\mu})\right| \int_G \left| \e(T) \right| \cdot \left|  \h(T) \right|  \dx \\
& \leq \frac{1}{2} \sup_{\x\in G} |\m(\x)| \left|\lambda_{\max}(\boldsymbol{\varepsilon}) \lambda_{\max}(\boldsymbol{\mu})\right| \int_G \left| \e(T) \right|^2 +  \left|  \h(T) \right|^2  \dx \\
& \leq \sup_{\x\in G} |\m(\x)| \left|\lambda_{\max}(\boldsymbol{\varepsilon}) \lambda_{\max}(\boldsymbol{\mu})\right| E_{\xi}(T).
\end{split}
\end{align}
Similarly, we obtain
\begin{align}
\label{STABILITY_EQ_10}
\left| \int_G \big( \m \times \boldsymbol{\varepsilon} \e(0) \big) \cdot \big(\boldsymbol{\mu} \h(0) \big)\dx \right| & \leq \sup_{\x\in G} |\m(\x)| \left|\lambda_{\max}(\boldsymbol{\varepsilon}) \lambda_{\max}(\boldsymbol{\mu})\right| E_{\xi}(0).
\end{align}

Next, we estimate $\int_0^T \int_\Gamma (\n \times \e)\cdot(\m \times \boldsymbol{\varepsilon} \e) \d\x\d t$. By virtue of Young's inequality, we get
\begin{equation}
\label{STABILITY_EQ_12}
|(\n \times \e)\cdot(\m \times \boldsymbol{\varepsilon} \e)| \leq |\n \times \e| \cdot  |\m \times \boldsymbol{\varepsilon} \e| \leq  \frac{1}{2\delta} |\n \times \e|^2 + \frac{\delta}{2}  |\m \times \boldsymbol{\varepsilon} \e|^2
\end{equation}
Using the uniform positive definiteness of $\boldsymbol{\varepsilon}$, we further find
\begin{align*}
|\m \times \boldsymbol{\varepsilon} \e |^2 &\leq \sup_{\x \in G} |\m(\x)|^2 \cdot  |\boldsymbol{\varepsilon} \e |^2 \leq  \sup_{\x \in G} |\m(\x)|^2  \big( \lambda_{\max}(\boldsymbol{\varepsilon}) |\e| \big)^2 \\
&\leq  \sup_{\x \in G} |\m(\x)|^2  \big( \lambda_{\max}(\boldsymbol{\varepsilon})  \big)^2 \frac{1}{\alpha} \boldsymbol{\varepsilon} \e \cdot \e.
\end{align*}
Integrating the latter inequality, we get
\begin{align}
\label{STABILITY_EQ_13}
\int_0^T \int_\Gamma |\m \times \boldsymbol{\varepsilon} \e|^2 \d\x\d t \leq  \frac{1}{\alpha}  \sup_{\x \in G} |\m(\x)|^2  \big( \lambda_{\max}(\boldsymbol{\varepsilon})  \big)^2 \int_0^T \int_\Gamma   \boldsymbol{\varepsilon} \e \cdot \e \d\x\d t.
\end{align}

Using Equations (\ref{STABILITY_EQ_12}) and (\ref{STABILITY_EQ_13}), we obtain
\begin{align}
\notag
\int_0^T \int_\Gamma \big( \n \times \e \big)\cdot \big( \m \times \boldsymbol{\varepsilon} \e \big) \ \d\x\d t &\leq \int_0^T \int_\Gamma\frac{1}{2\delta} |\n \times \e|^2 + \frac{\delta}{2}  |\m \times \boldsymbol{\varepsilon} \e|^2 \ \d\x\d t \\
\label{STABILITY_EQ_17}
&\leq \frac{1}{2\delta} \int_0^T \int_\Gamma |\n \times \e|^2 \d\x \d t \\
\notag
&+ \frac{\delta}{2} \frac{1}{\alpha}  \sup_{\x \in G} |\m(\x)|^2  \big( \lambda_{\max}(\boldsymbol{\varepsilon})  \big)^2 \int_0^T \int_\Gamma   \boldsymbol{\varepsilon} \e \cdot \e \d\x\d t.
\end{align}
In the same fashion, we get
\begin{align}
\notag
\int_0^T \int_\Gamma (\n \times \h)\cdot(\m \times \boldsymbol{\mu} \h) \d\x\d t &\leq \int_0^T \int_\Gamma\frac{1}{2\delta} |\n \times \h|^2 + \frac{\delta}{2}  |\m \times \boldsymbol{\mu} \h |^2 \ \d\x\d t \\
\label{STABILITY_EQ_16}
&\leq \frac{1}{2\delta} \int_0^T \int_\Gamma |\n \times \h|^2 \d\x \d t \\
\notag
&+ \frac{\delta}{2} \frac{1}{\alpha}  \sup_{\x \in G} |\m(\x)|^2  \big( \lambda_{\max}(\boldsymbol{\mu})  \big)^2 \int_0^T \int_\Gamma   \boldsymbol{\mu} \h \cdot \h \d\x\d t.
\end{align}

Recalling the boundary condition in Equation (\ref{EQUATION_MAXWELL_BC_1}), we estimate
\begin{align}
\notag
 \frac{1}{2\delta} \int_0^T \int_\Gamma |\n \times \h|^2 \d\x \d t = & \frac{1}{2\delta} \int_0^T \int_\Gamma |\big( \gamma_1 \g(\zz{0}) + \gamma_2 \g(\zz{1}) \big) \times \n |^2 \d\x \d t \\
\notag
 \leq & \frac{1}{2\delta} \int_0^T \int_\Gamma |\big( \gamma_1 \g(\zz{0}) + \gamma_2 \g(\zz{1}) \big) |^2 \d\x \d t \\
\notag
 \leq & \frac{\max\{\gamma_1^2, \gamma_2^2\}}{\delta} \int_0^T \int_\Gamma  \g^2(\zz{0}) + \g^2(\zz{1}) \d\x \d t \\
\label{STABILITY_EQ_20}
 \leq & \frac{c_2^2 \max\{\gamma_1^2, \gamma_2^2\}}{\delta} \int_0^T \int_\Gamma  \zz{0}^2 + \zz{1}^2 \d\x \d t.
\end{align}

Combining Equations (\ref{STABILITY_EQ_15})--(\ref{STABILITY_EQ_10}) and  (\ref{STABILITY_EQ_17})--(\ref{STABILITY_EQ_20}), we deduce
\begin{align}
\notag
\frac{d_1 \alpha}{2} \int_{0}^{T} \int_{G} \left| \e \right|^2  +  \left| \h \right|^2 \d\x \d t \leq &  -\frac{ \beta }{2}\int_0^T \int_\Gamma \boldsymbol{\mu}\h \cdot \h + \boldsymbol{\varepsilon} \e \cdot \e  \d\x\d t  \\
\notag
& + \frac{1}{2\delta} \int_0^T \int_\Gamma |\n \times \e|^2 \d\x \d t \\
\notag
& +  \frac{\delta}{2} \cdot \frac{1}{\alpha}  \sup_{\x \in G} |\m(\x)|^2  \big( \lambda_{\max}(\boldsymbol{\varepsilon})  \big)^2 \int_0^T \int_\Gamma   \boldsymbol{\varepsilon} \e \cdot \e \d\x\d t \\
\notag
& + \frac{c_2^2 \max\{\gamma_1^2, \gamma_2^2\}}{\delta} \int_0^T \int_\Gamma  \zz{0}^2 + \zz{1}^2 \d\x \d t \\
\notag
& +  \frac{\delta}{2} \frac{1}{\alpha}  \sup_{\x \in G} |\m(\x)|^2  \big( \lambda_{\max}(\boldsymbol{\mu})  \big)^2 \int_0^T \int_\Gamma   \boldsymbol{\mu} \h \cdot \h \d\x\d t  \\
\notag
& + \sup_{\x\in G} |\m(\x)| \left|\lambda_{\max}(\boldsymbol{\varepsilon}) \lambda_{\max}(\boldsymbol{\mu})\right| (E_{\xi}(T)+E_{\xi}(0)).
\end{align}

Choosing $\delta > 0$ such that
\begin{equation}
\frac{\delta}{2} \frac{1}{\alpha}  \sup_{\x \in G} |\m(\x)|^2 \max\big\{ \big( \lambda_{\max}(\boldsymbol{\varepsilon})  \big)^2 , \big( \lambda_{\max}(\boldsymbol{\mu}) \big)^2 \big\}\leq \frac{\beta}{2},
\end{equation}
we arrive at
\begin{align}
\label{STABILITY_EQ_18}
\frac{d_1 \alpha}{2} \int_{0}^{T} \int_{G} \left| \e \right|^2  +  \left| \h \right|^2 \d\x \d t \leq &  \frac{1}{2\delta} \int_0^T \int_\Gamma \zz{0}^2 \d\x \d t \\
\notag
& + \frac{c_2^2 \max\{\gamma_1^2, \gamma_2^2\}}{\delta}  \int_0^T \int_\Gamma   \zz{0}^2 +  \zz{1}^2  \ \d\x\d t \\
\notag & + \sup_{\x\in G} |\m(\x)|
\left|\lambda_{\max}(\boldsymbol{\varepsilon})
\lambda_{\max}(\boldsymbol{\mu})\right|
\big(E_{\xi}(T)+E_{\xi}(0)\big).
\end{align}

There remains to estimate
\begin{equation*}
I = \int_0^T \int_{0}^{1} \int_{\Gamma} \big|\mathbf{E}(t - \tau s, \mathbf{x}) \times \boldsymbol{\nu}\big|^{2} \dx \d s \d t.
\end{equation*}
Making substitution $u = t-\tau s$ and $v = t$, we get
\begin{align}
\notag
I  &= \frac{1}{\tau}  \int_0^T \int_{v-\tau}^{v} \int_{\Gamma} \big|\mathbf{E}(u, \mathbf{x}) \times \boldsymbol{\nu}\big|^{2} \dx \d s \d v \\
\notag
&= \frac{1}{\tau} \int_{-\tau}^0 (u+\tau)\int_{\Gamma} \big|\mathbf{E}(u, \mathbf{x}) \times \boldsymbol{\nu}\big|^{2} \dx \d u
 + \frac{1}{\tau} \int_0^{T-\tau} \tau \int_{\Gamma} \big|\mathbf{E}(u, \mathbf{x}) \times \boldsymbol{\nu}\big|^{2} \dx \d u \\
\notag
&+ \frac{1}{\tau} \int_{T-\tau}^T (T- u) \int_{\Gamma} \big|\mathbf{E}(u, \mathbf{x}) \times \boldsymbol{\nu}\big|^{2} \dx \d u \\
\notag
&\leq \int_{-\tau}^0 \int_{\Gamma} \big|\mathbf{E}(u, \mathbf{x}) \times \boldsymbol{\nu}\big|^{2} \dx \d u+ \int_0^{T-\tau} \int_{\Gamma} \big|\mathbf{E}(u, \mathbf{x}) \times \boldsymbol{\nu}\big|^{2} \dx \d u \\
\notag
&+ \int_{T-\tau}^T \int_{\Gamma} \big|\mathbf{E}(u, \mathbf{x}) \times \boldsymbol{\nu}\big|^{2} \dx \d u \\
\notag
&\leq \int_0^T \int_{\Gamma} \big|\mathbf{E}(u, \mathbf{x}) \times \boldsymbol{\nu}\big|^{2} \dx \d u +  \int_0^T \int_{\Gamma} \big|\mathbf{E}(u-\tau, \mathbf{x}) \times \boldsymbol{\nu}\big|^{2} \dx \d u \\
\label{STABILITY_EQ_19}
&= \int_0^T \int_{\Gamma} \zz{0}^{2}  + \zz{1}^{2} \dx \d t
\end{align}
Now, multiplying Equation (\ref{STABILITY_EQ_19}) by $\xi \tau$ and adding the result to Equation (\ref{STABILITY_EQ_18}) divided by $d_1 \alpha$,
the claim follows with appropriate constants $c, c_T$.
\end{proof}

\begin{theorem}
    \label{THEOREM_EXPONENTIAL_STABILITY}

    Let $\mathbf{V}$ be the unique strong solution given in Theorem \ref{THEOREM_WELL_POSEDNESS}.
    Under Assumption \ref{ASSUMPTION_GEOMETRY}, if $c_1\gamma_{1} > c_2\gamma_{2}$ (i.e., the delay term is not too strong),
    there exist $C, \lambda > 0$ such that the associated energy satisfies
    \begin{equation}
        E(t) \leq C e^{-\lambda t} E(0) \text{ for } t \geq 0. \notag
    \end{equation}
\end{theorem}
\begin{proof}
From Lemmas \ref{STABILITY_ENERGY_DUMPING_INEQUALITY}, \ref{STABILITY_LEMMA_2} and Theorem \ref{APPENDIX_THEOREM} in the appendix with
\begin{align*}
D(t) = \int_{\Gamma} \zz{0}^2 + \zz{1}^2\ \d\x,
\end{align*}
we get the desired inequality for $E_{\xi}(\cdot)$ in place of $E(\cdot)$.
Taking into account the equivalence of $E(\cdot)$ and $E_{\xi}(\cdot)$, the original claim follows.
\end{proof}

Due to the density of $D(\mathscr{A})$ in $\mathscr{H}$, we have:
\begin{corollary}
    The conclusions of Theorem \ref{THEOREM_EXPONENTIAL_STABILITY} remain true for mild solutions, i.e., if $\mathbf{V}^{0} \in \mathscr{H}$.
\end{corollary}

\section*{Acknowledgment}
Financial support by the Deutsche Forschungsgemeinschaft (DFG) -- through CRC 1173 at Karlsruhe Institute of Technology, Germany -- and the University of Texas at El Paso is gratefully acknowledged.

\bibliographystyle{plain}
\bibliography{bibliography}

\appendix

\section{Proof of Exponential Stability}
\label{SECTION_APPENDIX}

\begin{theorem}
\label{APPENDIX_THEOREM}
Suppose there exist a non-negative function $D(t)$ and positive numbers $c_1^E, c_2^E, c$ and $c_T$ such that
\begin{equation}
\label{APPENDIX_INEQ_1}
- c_1^E \int_{t_1}^{t_2} D(t)\ \d t \geq E(t_2) - E(t_1) \geq - c_2^E \int_{t_1}^{t_2} D(t)\ \d t \text{ for all } t_2>t_1 \geq 0
\end{equation}
and
\begin{equation}
\label{APPENDIX_INEQ_2}
\int_{0}^{T} E (t)\ \d t \leq c (E(0) + E(T)) + c_T \int_{0}^{T} D(t)\ \d t \text{ for arbitrarily large } T.
\end{equation}
Then, there exist $C, \lambda > 0$ such that the function $E(t)$ satisfies
\begin{equation}
      E(t) \leq C e^{-\lambda t} E(0) \text{ for } t \geq 0. \notag
\end{equation}
\end{theorem}

\begin{proof}
Taking $t_1 = 0$ and $t_2 = T$ in Equation (\ref{APPENDIX_INEQ_1}), we get
\begin{equation}
\label{APPENDIX_INEQ_3}
E(0)  \leq E(T) +  c_2^E \int_{0}^{T} D(t)\ \d t.
\end{equation}
Thus, from Equation (\ref{APPENDIX_INEQ_2}), we obtain
\begin{equation}
\label{APPENDIX_INEQ_4}
\int_0^T E(t)\ \d t  \leq  2c  E(T) + (c_T+cc_2^E) \int_{0}^{T} D(t)\ \d t.
\end{equation}
Now, using Equation (\ref{APPENDIX_INEQ_1}) with $t_2 = T$ and $t_1 = t$, we get
\begin{equation}
\label{APPENDIX_INEQ_5}
E(t)  \geq E(T)  +  c_1^E \int_{t}^{T} D(s)\ \d s.
\end{equation}

Integrating the latter inequality from $0$ to $T$ with respect to $t$ and taking into account Equation (\ref{APPENDIX_INEQ_4}), we arrive at
\begin{equation}
\label{APPENDIX_INEQ_6}
T E(T)  +  c_1^E  \int_0^T \int_{t}^{T} D(s)\ \d s \d t \leq  \int_0^T E(t)\ \d t \leq   2c  E(T) + (c_T+cc_2^E) \int_{0}^{T} D(t)\ \d t.
\end{equation}
Choosing $T>4c$, we have
\begin{equation}
\label{APPENDIX_INEQ_7}
\frac{T}{2} E(T)  +  c_1^E  \int_0^T \int_{t}^{T} D(s)\ \d s \d t \leq  (c_T+cc_2^E) \int_{0}^{T} D(t)\ \d t.
\end{equation}
Since $D(s)$ is non-negative, we estimate
\begin{equation}
\label{APPENDIX_INEQ_8}
\frac{T}{2} E(T)  \leq  (c_T+cc_2^E) \int_{0}^{T} D(t)\ \d t.
\end{equation}
Applying Equation (\ref{APPENDIX_INEQ_1}) with $t_1 = 0$ and $t_2 = T$ to the inequality in Equation (\ref{APPENDIX_INEQ_8}), we get
\begin{equation}
\label{APPENDIX_INEQ_9}
\frac{T}{2} E(T)  \leq  \frac{c_T+cc_2^E}{c_1^E} (E(0) - E(T)),
\end{equation}
which finally leads us to
\begin{equation}
\label{APPENDIX_INEQ_10}
\left( \frac{T}{2} + \tilde{c} \right) E(T)  \leq  \tilde{c} E(0)
\end{equation}
with $\tilde{c} = \frac{c_T+cc_2^E}{c_1^E}$.
Thus,
\begin{equation}
\label{APPENDIX_INEQ_11}
 E(T)  \leq  \gamma E(0) \text{ for } \gamma = \frac{\tilde{c}}{\tilde{c}+T/2} < 1.
\end{equation}

Using a similar argument on each of the time segments $[(m-1)T,mT]$ for $m=1,2,\dots$, we  obtain
\begin{equation}
\label{APPENDIX_INEQ_12}
 E(mT)  \leq  \gamma E((m-1)T) \leq \ldots \leq \gamma^{m} E(0),\ m = 1,2,\ldots
\end{equation}
Denoting $\lambda = - T^{-1} \ln(\gamma) > 0$, Equation (\ref{APPENDIX_INEQ_12}) rewrites as
\begin{equation}
\label{APPENDIX_INEQ_13}
 E(mT)  \leq e^{-\lambda m T} E(0),\ m = 1,2, \ldots.
\end{equation}
It easily follows from (\ref{APPENDIX_INEQ_1}) that $E(t)$ is monotone non-increasing. This leads to
\begin{equation}
\label{APPENDIX_INEQ_14}
E(t) \leq  E(mT)  \leq e^{-\lambda m T} E(0) = \frac{1}{\gamma} e^{-\lambda (m+1) T} E(0) \leq \frac{1}{\gamma} e^{-\lambda t} E(0)
\end{equation}
for arbitrary $t\in [mT, (m+1)T]$ for any $m = 1, 2, \dots$, which completes the proof.
\end{proof}

\end{document}